\documentclass[final,onefignum,onetabnum]{siamart250211}

\usepackage{amsmath,amsfonts,amssymb,array,graphicx,mathtools,multirow,bm,tcolorbox,relsize,booktabs,dsfont}
\usepackage{mathrsfs}
\usepackage{multirow}
\usepackage{arydshln} 
\allowdisplaybreaks
\usepackage{stmaryrd} 
\usepackage{lipsum}
\usepackage{amsfonts}
\usepackage{graphicx}
\usepackage{epstopdf}
\usepackage{algorithmic}

\ifpdf
  \DeclareGraphicsExtensions{.eps,.pdf,.png,.jpg}
\else
  \DeclareGraphicsExtensions{.eps}
\fi

\newsiamremark{remark}{Remark}
\newsiamremark{hypothesis}{Hypothesis}
\crefname{hypothesis}{Hypothesis}{Hypotheses}
\newsiamthm{claim}{Claim}

\usepackage{tikz}
\usepackage{tikz-3dplot}
\usetikzlibrary{patterns}
\usetikzlibrary{shapes,calc,shapes,arrows}

\headers{Normalized tensor train decomposition}{R. Peng, C. Zhu, B. Gao, X. Wang, and Y.-x. Yuan}

\title{Normalized tensor train decomposition\thanks{Submitted to the editors DATE.
\funding{This work was supported by the National Key R\&D Program of China (grant 2023YFA1009300). XW was supported by the Guangdong Provincial Quantum Science Strategic Initiative (grant No.~GDZX2403008) and the Guangdong Provincial Key Lab of Integrated Communication, Sensing and Computation for Ubiquitous Internet of Things (grant No.~2023B1212010007). BG and YY were supported by the National Natural Science Foundation of China (grant No.~12288201).}}}

\author{Renfeng Peng\thanks{State Key Laboratory of Mathematical Sciences, Academy of Mathematics and Systems Science, Chinese  Academy of Sciences, and University of Chinese Academy of Sciences, China ({pengrenfeng@lsec.cc.ac.cn}).}
    \and Chengkai Zhu\footnotemark[4]
    \and Bin Gao\thanks{State Key Laboratory of Mathematical Sciences, Academy of Mathematics and Systems Science, Chinese  Academy of Sciences, China
		({\{gaobin,yyx\}@lsec.cc.ac.cn}).}
    \and Xin Wang\thanks{Thrust of Artificial Intelligence, Information Hub, The Hong Kong University of Science and Technology (Guangzhou), China ({czhu696@connect.hkust-gz.edu.cn;felixxinwang@hkust-gz.edu.cn})}
	\and Ya-xiang Yuan\footnotemark[3]
}

\usepackage{amsopn}

\usepackage{amssymb}

\newcommand{\vecx}{\mathbf{x}}
\newcommand{\vecr}{\mathbf{r}}
\newcommand{\vecu}{\mathbf{u}}
\newcommand{\vecv}{\mathbf{v}}

\newcommand{\mata}{\mathbf{A}}
\newcommand{\matb}{\mathbf{B}}
\newcommand{\matc}{\mathbf{C}}

\newcommand{\matx}{\mathbf{X}}
\newcommand{\matu}{\mathbf{U}}
\newcommand{\matv}{\mathbf{V}}
\newcommand{\matW}{\mathbf{W}}

\newcommand{\matG}{\mathbf{G}}
\newcommand{\matI}{\mathbf{I}}
\newcommand{\matH}{\mathbf{H}}

\newcommand{\matK}{\mathbf{K}}

\newcommand{\matP}{\mathbf{P}}

\newcommand{\matt}{\mathbf{T}}
\newcommand{\matT}{\mathbf{T}}

\newcommand{\matS}{\mathbf{S}}

\newcommand{\matU}{\mathbf{U}}

\newcommand{\matY}{\mathbf{Y}}

\newcommand{\tensA}{\cA}
\newcommand{\tensB}{\mathcal{B}}

\newcommand{\tensE}{\mathcal{E}}

\newcommand{\tensG}{\mathcal{G}}

\newcommand{\tensM}{\mathcal{M}}
\newcommand{\tensN}{\mathcal{N}}

\newcommand{\tensT}{\mathcal{T}}
\newcommand{\tensU}{\mathcal{U}}
\newcommand{\tensV}{\mathcal{V}}
\newcommand{\tensW}{\mathcal{W}}
\newcommand{\tensY}{\mathcal{Y}}
\newcommand{\tensX}{\mathcal{X}}

\newcommand{\ranktt}{\mathrm{rank}_{\mathrm{TT}}}
\newcommand{\rmvec}{\mathrm{vec}}
\newcommand{\St}{\mathrm{St}}

\newcommand{\subjectto}{\mathrm{s.\,t.}}

\newcommand{\frob}{\mathrm{F}}

\newcommand{\leftunfolding}{\mathrm{L}}
\newcommand{\rightunfolding}{\mathrm{R}}

\DeclareMathOperator{\conj}{conj}

\DeclareMathOperator{\tangent}{T}

\DeclareMathOperator{\retr}{R}

\DeclareMathOperator{\proj}{P}

\DeclareMathOperator*{\argmin}{arg\,min}

\newcommand{\STAB}{\mathrm{STAB}_n}
\newcommand{\Cl}{\mathrm{Cl}_n}
\newcommand{\nc}{\newcommand}
\nc{\rnc}{\renewcommand}
\nc{\lbar}[1]{\overline{#1}}
\nc{\bra}[1]{\langle#1|}
\nc{\ket}[1]{|#1\rangle}
\nc{\dketbra}[2]{\vert #1 \rangle \hspace{-.8mm} \rangle \hspace{-.4mm} \langle\hspace{-.8mm}\langle #2 \vert}
\nc{\dbra}[1]{\langle\hspace{-.8mm}\langle #1\vert}
\nc{\dket}[1]{\vert#1\rangle\hspace{-.8mm}\rangle}
\nc{\ketbra}[2]{|#1\rangle\!\langle#2|}
\nc{\braket}[2]{\langle#1|#2\rangle}
\nc{\braandket}[3]{\langle #1|#2|#3\rangle}

\nc{\avg}[1]{\langle#1\rangle}
\nc{\rank}{\operatorname{rank}}
\nc{\smfrac}[2]{\mbox{$\frac{#1}{#2}$}}
\nc{\tr}{\operatorname{Tr}}
\nc{\ox}{\otimes}
\nc{\dg}{\dagger}
\nc{\dn}{\downarrow}
\nc{\cA}{{\cal A}}
\nc{\cB}{{\cal B}}
\nc{\cC}{{\cal C}}
\nc{\cD}{{\cal D}}
\nc{\cE}{{\cal E}}
\nc{\cF}{{\cal F}}
\nc{\cG}{{\cal G}}
\nc{\cH}{{\cal H}}
\nc{\cI}{{\cal I}}
\nc{\cJ}{{\cal J}}
\nc{\cK}{{\cal K}}
\nc{\cL}{{\cal L}}
\nc{\cM}{{\cal M}}
\nc{\cN}{{\cal N}}
\nc{\cO}{{\cal O}}
\nc{\cP}{{\cal P}}
\nc{\cQ}{{\cal Q}}
\nc{\cR}{{\cal R}}
\nc{\cS}{{\cal S}}
\nc{\cT}{{\cal T}}
\nc{\cU}{{\cal U}}
\nc{\cV}{{\cal V}}
\nc{\cX}{{\cal X}}
\nc{\cY}{{\cal Y}}
\nc{\cZ}{{\cal Z}}
\nc{\cW}{{\cal W}}
\nc{\csupp}{{\operatorname{csupp}}}
\nc{\qsupp}{{\operatorname{qsupp}}}
\nc{\var}{{\operatorname{var}}}
\nc{\rar}{\rightarrow}
\nc{\lrar}{\longrightarrow}
\nc{\polylog}{{\operatorname{polylog}}}
\nc{\idop}{{\mathds{1}}}
\nc{\wt}{{\operatorname{wt}}}
\nc{\av}[1]{{\left\langle {#1} \right\rangle}}
\nc{\supp}{{\operatorname{supp}}}
\nc{\SEP}{{\mathrm{SEP}}}
\nc{\PPT}{{\mathrm{PPT}}}

\nc{\RR}{{{\mathbb R}}}
\nc{\CC}{{{\mathbb C}}}
\nc{\FF}{{{\mathbb F}}}
\nc{\NN}{{{\mathbb N}}}
\nc{\ZZ}{{{\mathbb Z}}}
\nc{\PP}{{{\mathbb P}}}
\nc{\QQ}{{{\mathbb Q}}}
\nc{\UU}{{{\mathbb U}}}
\nc{\EE}{{{\mathbb E}}}
\nc{\id}{{\operatorname{id}}}
\nc{\LOCC}{{\text{\rm LOCC}}}

\begin{document}

\maketitle

\begin{abstract}
Tensors with unit Frobenius norm are fundamental objects in many fields, including scientific computing and quantum physics, which are able to represent normalized eigenvectors and pure quantum states. While the tensor train decomposition provides a powerful low-rank format for tackling high-dimensional problems, it does not intrinsically enforce the unit-norm constraint. To address this, we introduce the normalized tensor train (NTT) decomposition, which aims to approximate a tensor by unit-norm tensors in tensor train format. The low-rank structure of NTT decomposition not only saves storage and computational cost but also preserves the underlying unit-norm structure. We prove that the set of fixed-rank NTT tensors forms a smooth manifold, and the corresponding Riemannian geometry is derived, paving the way for geometric methods. We propose NTT-based methods for low-rank tensor recovery, high-dimensional eigenvalue problem, estimation of stabilizer rank, and calculation of the minimum output R\'enyi 2-entropy of quantum channels. Numerical experiments demonstrate the superior efficiency and scalability of the proposed NTT-based methods. 
\end{abstract}



\section{Introduction}
Tensors are higher-dimensional generalizations of matrices, i.e., a tensor is an array with $d$ indices, which provides powerful tools for representing high-dimensional datasets. However, storing a tensor in full size becomes prohibitively expensive, as the required memory grows exponentially with~$d$. By imposing a low-rank structure, one can capture the most essential information of a tensor and significantly reduce storage requirements. Low-rank tensors have demonstrated effectiveness in various applications, including image processing~\cite{vasilescu2003multilinear}, tensor completion~\cite{kressner2014low,gao2024riemannian}, high-dimensional eigenvalue problems~\cite{dolgov2014computation}, and high-dimensional partial differential equations~\cite{bachmayr2023low}; see~\cite{grasedyck2013literature} for an overview. The low-rank structure of a tensor depends on a specific tensor decomposition format. The canonical polyadic (CP) decomposition~\cite{hitchcock1928multiple}, Tucker decomposition~\cite{de2000multilinear}, hierarchical Tucker decomposition~\cite{uschmajew2013geometry}, and tensor train (TT) decomposition~\cite{oseledets2011tensor} (also known as matrix product states (MPS) in quantum physics~\cite{vidal2003efficient,Garcia2007}) are among the most typical formats. We refer to~\cite{kolda2009tensor} for an overview. 

A critical observation is that in many applications, e.g., in quantum physics, the tensors of interest are not only low-rank but also inherently normalized. Thus, we introduce the following \emph{normalized} tensor train (NTT) decomposition. Given a tensor $\tensA\in\CC^{n_1\times n_2\times\cdots\times n_d}$, the NTT decomposition of~$\tensA$ aims to approximate $\tensA$ by a low-rank tensor $\llbracket\tensU_1,\tensU_2,\dots,\tensU_d\rrbracket$ with unit Frobenius norm,
\[\tensA(i_1,i_2,\dots,i_d)\approx\matu_1(i_1)\matu_2(i_2)\cdots\matu_d(i_d)\quad\text{with}\quad\left\|\llbracket\tensU_1,\tensU_2,\dots,\tensU_d\rrbracket\right\|_\frob=1\]
for $i_k=1,2,\dots,n_k$ and $k=1,2,\dots,d$, where $\tensU_k\in\CC^{r_{k-1}\times n_k\times r_k}$ is a core tensor, $\matu_k(i_k)={\tensU_k(:,i_k,:)}$, and $r_0,r_1,\dots,r_d$ are positive integers with $r_0=r_d=1$. We refer to $\llbracket\tensU_1,\tensU_2,\dots,\tensU_d\rrbracket$ as a \emph{NTT tensor}. \cref{fig: TT}~depicts the NTT decomposition of a tensor. Note that NTT decomposition can also be defined for tensors in~$\mathbb{R}^{n_1\times n_2\times\cdots\times n_d}$.

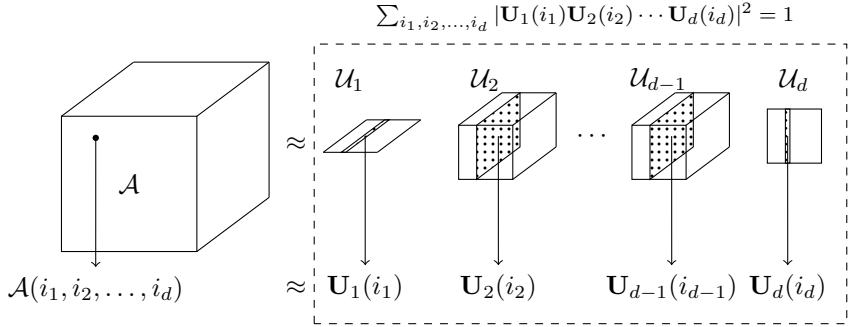
\begin{figure}[htbp]
	\centering
	\begin{tikzpicture}[scale = 1.2]
		\draw[-] (-1.4,1.6) -- (-0.6,2.2);
		\draw[-] (-0.6,2.2) -- (-2.1,2.2);
		\draw[-] (-2.1,2.2) -- (-2.9,1.6);
		\draw[-] (-2.9,1.6) rectangle (-1.4,0.1);
		\draw[-] (-1.4,0.1) -- (-0.6,0.7);
		\draw[-] (-0.6,0.7) -- (-0.6,2.2);
		\node (v) at (-2.15,0.85) {$\tensA$};

            \fill (-2.52,0.36+1) circle (1pt);
		\node (x123) at (-2.52,-0.32) {$\tensA(i_1,i_2,\dots,i_d)$};
		\draw[->] (-2.52,0.36+1) -- (x123);
	
		\node (u) at (-0.3,1.3) {$\approx$};
		\node (u) at (-0.3,-0.3) {$\approx$};
	
		\draw[-] (1.5-0.9,1.5-0.3) -- (1.98-0.9,1.86-0.3);
		\draw[-] (1.98-0.9,1.86-0.3) -- (1.38-0.9,1.86-0.3);
		\draw[-] (0.9-0.9,1.5-0.3) -- (1.38-0.9,1.86-0.3);
		\draw[-] (0.9-0.9,1.5-0.3) -- (1.5-0.9,1.5-0.3);
		\node (G) at (1.2-0.9,2) {$\tensU_1$};
		\path[draw, pattern=dots] (0.2,1.2) -- (0.25,1.2) -- (0.73,1.56) -- (0.68,1.56) -- cycle;
		\node (G1) at (0.465,-0.3) {$\matU_1(i_1)$};
		\draw[->] (0.465,1.38) -- (G1);
		
		\draw[-] (1.5+0.6,1.5) -- (1.98+0.6,1.86);
		\draw[-] (1.98+0.6,1.86) -- (1.38+0.6,1.86);
		\draw[-] (0.9+0.6,1.5) -- (1.38+0.6,1.86);
		\draw[-] (1.98+0.6,1.86) -- (1.98+0.6,1.26);
		\draw[-] (1.5+0.6,0.9) -- (1.98+0.6,1.26);
		\draw[-] (0.9+0.6,1.5) rectangle (1.5+0.6,0.9);
		\node (G) at (1.2+0.6,2) {$\tensU_{2}$};
		\path[draw, pattern=dots] (0.2+1.5,0.9) -- (0.68+1.5,1.26) -- (0.68+1.5,1.86) -- (0.2+1.5,1.5) -- cycle;
		\node (G2) at (0.44+1.5,-0.3) {$\matU_{2}(i_2)$};
		\draw[->] (0.44+1.5,1.38) -- (G2);
		
		\node(cdots) at (3,1.38) {$\cdots$};

		\draw[-] (1.5+0.6+1.92,1.5) -- (1.98+0.6+1.92,1.86);
		\draw[-] (1.98+0.6+1.92,1.86) -- (1.38+0.6+1.92,1.86);
		\draw[-] (0.9+0.6+1.92,1.5) -- (1.38+0.6+1.92,1.86);
		\draw[-] (1.98+0.6+1.92,1.86) -- (1.98+0.6+1.92,1.26);
		\draw[-] (1.5+0.6+1.92,0.9) -- (1.98+0.6+1.92,1.26);
		\draw[-] (0.9+0.6+1.92,1.5) rectangle (1.5+0.6+1.92,0.9);
		\node (G) at (1.2+0.6+1.92,2) {$\tensU_{d-1}$};
		\path[draw, pattern=dots] (0.2+1.5+1.92,0.9) -- (0.68+1.5+1.92,1.26) -- (0.68+1.5+1.92,1.86) -- (0.2+1.5+1.92,1.5) -- cycle;
		\node (Gd1) at (0.44+1.5+1.92,-0.3) {$\matU_{d-1}(i_{d-1})$};
		\draw[->] (0.44+1.5+1.92,1.38) -- (Gd1);

		\draw[-] (0.9+0.6+1.92+1.5,1.5+0.18) rectangle (1.5+0.6+1.92+1.5,0.9+0.18);
		\node (G) at (1.2+0.6+1.92+1.5,2) {$\tensU_{d}$};
		\path[draw, pattern=dots] (0.2+1.5+1.92+1.5,0.9+0.18) -- (0.25+1.5+1.92+1.5,0.9+0.18) -- (0.25+1.5+1.92+1.5,1.5+0.18) -- (0.2+1.5+1.92+1.5,1.5+0.18) -- cycle;
		\node (Gd) at (0.225+1.5+1.92+1.5,-0.3) {$\matU_{d}(i_{d})$};
		\draw[->] (0.225+1.5+1.92+1.5,1.38) -- (Gd);

            \draw[dashed] (-0.1,-0.7) rectangle (5.8,2.4);
            \node at (2.9,2.7) {\footnotesize$\sum_{i_1,i_2,\dots,i_d}|\matu_1(i_1)\matu_2(i_2)\cdots\matu_d(i_d)|^2=1$}; 
	\end{tikzpicture}
        \caption{Normalized tensor train decomposition of a tensor.}
        \label{fig: TT}
    \end{figure}

\subsection{Applications}
We present several applications of the NTT decomposition. The first class of applications arises in scientific computing.

\paragraph{Low-rank tensor recovery} 
Given a partially observed unit-norm tensor $\tensA \in \mathbb{R}^{n_1\times n_2 \times \cdots \times n_d}$ in the NTT format on an index set $\Omega \subset [n_1]\times [n_2]\times \cdots \times [n_d]$, we aim to recover the full tensor by exploiting the low-rank structure of $\tensA$. This task arises in a variety of applications such as statistics, machine learning, and compression of high-dimensional functions; see, e.g.,~\cite{chu2005low,khoo2021efficient}. Specifically, the tensor recovery can be implemented by solving the following optimization problem on NTT tensors,
\begin{equation*}
    \begin{aligned}
        \min\quad & f(\tensX)=\frac12\|\proj_{\Omega}(\tensX)-\proj_\Omega(\tensA)\|_\frob^2\\
        \subjectto\quad & \tensX~\text{is a NTT tensor};
    \end{aligned}
\end{equation*}
see~\cref{subsec:lrtc} for details.

\paragraph{Eigenvalue problem with tensor product structure}
The computation of the smallest (largest) eigenvalue $\lambda_{\min}$ ($\lambda_{\max}$) and corresponding eigenvector $\vecx\in\mathbb{R}^{n_1n_2\cdots n_d}$ for a symmetric matrix $\mata\in\mathbb{R}^{(n_1n_2\cdots n_d)\times(n_1n_2\cdots n_d)}$ is one of the key problems in numerical linear algebra and computational physics~\cite{jones2015density}, where the space $\mathbb{R}^{n_1n_2\cdots n_d}$ arises, for example, from the discretization of a high-dimensional partial differential equation on tensor product space $\mathbb{R}^{n_1}\otimes\mathbb{R}^{n_2}\otimes\cdots\otimes\mathbb{R}^{n_d}$. Directly solving the problem suffers from the curse of dimensionality. Similarly, such an issue also appears in quantum many-body physics, where the \emph{Hamiltonian} of a quantum system is modeled by a Hermitian matrix in $\mathbb{C}^{(n_1n_2\cdots n_d)\times(n_1n_2\cdots n_d)}$. Finding the ground state of the Hamiltonian and its energy can also be interpreted by the eigenvalue problem. Moreover, the low-rank MPS can faithfully represent the ground state of local Hamiltonians~\cite{Verstraete2006}. In light of this observation, we aim to find a low-rank solution to the eigenvalue problem in the NTT format, i.e., 
\begin{equation*}
    \begin{aligned}
        \min_{\tensX}(\max_{\tensX})\quad & \rmvec(\tensX)^\top\mata\rmvec(\tensX)\\
        \subjectto\quad\quad &\tensX~\text{is a NTT tensor};
    \end{aligned}
\end{equation*}
see~\cref{subsec:eig_prob} for details. 

The second class of applications is motivated by quantum information theory, where the NTT decomposition allows for efficient numerical investigation of two fundamental concepts: the quantification of quantum resources and the additivity of channel capacities.

\paragraph{Approximation of the stabilizer rank} We consider the non-stabilizerness, or the magic~\cite{Bravyi_2005}, a crucial ingredient for achieving quantum advantage.
A central measure of the nonstabilizerness for a pure state is the \textit{stabilizer rank}~\cite{Bravyi2016}, i.e., the smallest integer $R$ for which the target state can be written as a convex combination of $R$ stabilizer states. However, computing the stabilizer rank of a given $n$-qubit state is intractable, as the number of stabilizer states grows super-exponentially. Consequently, any brute-force approach (e.g., searching through all possible tuples of stabilizer states to find a decomposition) is computationally prohibitive, even for small systems. Therefore, we introduce the notion of $(\epsilon,\delta)$-approximate stabilizer rank for a pure state $\ket\psi$, and propose evaluating it by solving the following optimization problem on the Cartesian product of the set of fixed-rank NTT tensors, 
\begin{equation*}
\begin{aligned}
    \min_{\{c_j\}_j,\{\ket{\phi_j}\}_j} &\; \frac12\Big\|\sum_{j=1}^{R} c_j \ket{\phi_j} - \ket{\psi}\Big\|_\frob^2 + \sum_{j=1}^R M_2(\ket{\phi_j}) \\
    \subjectto\quad&\;\; c_1,\dots,c_R \in \CC,~ \text{each}~ \ket{\phi_j}~\text{is a NTT tensor};
\end{aligned}
\end{equation*}
see~\cref{subsec:stab_rank} for details.

\paragraph{Minimum output R\'enyi $p$-entropy} We utilize the NTT decomposition to investigate the additivity of the minimum output R\'enyi $p$-entropy of quantum channels, completely positive and trace-preserving linear maps. The non-additivity of this quantity when $p$ goes to 1, proven by Hastings in high dimensions~\cite{Hastings_2009}, resolved a major open problem in quantum information theory by implying the non-additivity of classical capacity of a quantum channel. However, finding explicit counterexamples remains a significant challenge. The primary bottleneck for examining the superadditivity is to compute the min-output entropy for $n$ tensor products of a channel, where the optimization is performed over the space of input quantum states, a sphere whose dimension grows exponentially with $n$. By representing the high-dimensional input state in the NTT format, we transform the problem of computing the minimum output entropy into a tractable optimization on the set of fixed-rank NTT tensors,
\begin{equation*}
    \begin{aligned}
        \min_{\ket{\psi}} &\quad \frac{1}{1-p}\log\tr(\cN^{\ox n}_{A\to B}(\ketbra{\psi}{\psi}_{A^n})^p)\\ 
        \ket{\psi}~\text{is a NTT tensor};
    \end{aligned}
\end{equation*}
see~\cref{subsec:Renyi} for details.

\subsection{Related work and motivation}
We provide an overview of the existing matrix and tensor decompositions and the related geometries. For low-rank matrices, the set of fixed-rank matrices $\{\matx\in\mathbb{R}^{m\times n}:\rank(\matx)=r\}$ is a smooth embedded manifold; see, e.g., \cite{helmke1995critical}. Beyond the fixed-rank scenario, Cason et al.~\cite{cason2013iterative} studied the matrix variety $\{\matx\in\mathbb{R}^{m\times n}:\rank(\matx)\leq r\}$ and developed an explicit parametrization of tangent cones. Since the matrix variety is non-smooth, a desingularization approach was developed~\cite{khrulkov2018desingularization,rebjock2024optimization}, where slack variables are introduced to construct a smooth manifold embedded in a higher-dimensional space. The developed geometries pave the way for geometric methods for the minimization of a smooth function on low-rank matrices~\cite{vandereycken2013low,schneider2015convergence}.

In contrast with the matrix rank, the rank of a tensor depends on the choice of tensor decomposition format. Uschmajew and Vandereycken~\cite{uschmajew2013geometry} developed the differential geometry of tensors in hierarchical Tucker format. Holtz et al.~\cite{holtz2012manifolds} proved that the set of fixed rank tensors in the TT format forms a smooth manifold and provided an explicit representation of tangent spaces. Variants such as the block TT decomposition are beneficial for large-scale eigenvalue computations; see~\cite{dolgov2014computation}. Recently, a desingularization approach was proposed in~\cite{gao2024desingularization} for bounded-rank tensors in the TT format. We refer to~\cite{steinlechner2016riemannian,uschmajew2020geometric} for geometric methods for minimization of a smooth function on low-rank tensors.

Several works investigated the geometry of low-rank matrices under additional constraints. For instance, Cason et al.~\cite{cason2013iterative} studied the geometry of the set of low-rank matrices with unit Frobenius norm. Rakhuba and Oseledets~\cite{rakhuba2018jacobi} considered computing the smallest eigenvalue on the set of fixed-rank matrices with unit-norm constraints. For computing more than one eigenvalue, Krumnow et al.~\cite{krumnow2021computing} proposed a trace minimization approach on the intersection of the Stiefel manifold and low-rank matrices. More recently, Yang et al.~\cite{yang2025space} analyzed the geometry of low-rank matrix varieties under orthogonally-invariant constraints. 

Due to the intricacy of tensors, the properties and geometry of normalized tensor train decomposition can not be generalized straightforwardly from existing results. 
First, the additional unit-norm constraint fundamentally changes the low-rank approximation problem. For TT decomposition, a \emph{quasi-optimal} low-rank approximation can be computed via sequential SVDs. However, due to the additional unit-norm constraint, it is no longer clear how to design a quasi-optimal approximation in the NTT format.
Second, it is well-known that fixed-rank tensors in the TT format form a smooth manifold. However, whether this property also holds for fixed-rank tensors in the NTT format remains unknown.
Third, projection onto the set of fixed-rank tensors in the NTT format requires additional procedures to enforce both the unit-norm and low-rank constraints, which inevitably increases computational cost, highlighting the need for efficient implementations of basic operations in the NTT format.

\subsection{Contributions}
In this paper, we propose the normalized tensor train decomposition, and delve into the properties and geometry of tensors in the NTT format. First, we prove that the NTT decomposition exists for tensors with unit Frobenius norm, which is equivalent to the TT decomposition. For a tensor $\tensA$ not having unit norm, we construct a rank-$\vecr$ approximation operator $\proj^{\mathrm{NTTSVD}}_\vecr$ via TT-SVD and projection onto the unit sphere, which enjoys quasi-optimality
\[\|\proj^{\mathrm{NTTSVD}}_\vecr(\tensA)-\tensA\|_\frob\leq (2\sqrt{d-1}+1)\|\proj_{\tensN_\vecr}(\tensA)-\tensA\|_\frob,\]
where $\proj_{\tensN_\vecr}(\tensA)$ is the best rank-$\vecr$ approximation in the NTT format. 

Subsequently, we consider the set 
\[\tensN_\vecr=\{\tensX\in\CC^{n_1\times n_2\times\cdots\times n_d}:\ranktt(\tensX)=\vecr\ \text{and}\ \|\tensX\|_\frob=1\}\] 
of rank-$\vecr$ tensors in the NTT format, which is the intersection of the manifold of fixed-rank tensors in the TT format and the unit sphere. Since two smooth manifolds intersect transversally, we deduce that $\tensN_\vecr$ is a smooth manifold. We develop the Riemannian geometry of $\tensN_\vecr$, facilitating the geometric methods on $\tensN_\vecr$. The low-rank structure of $\tensN_\vecr$ not only saves storage and computational cost but also preserves the underlying unit-norm structure. 
The differences between tensor train decomposition and normalized TT decomposition are summarized in~\cref{tab:TTvsNTT}.
\begin{table}[htbp]
    \centering\footnotesize
    \caption{The differences between tensor train decomposition and normalized TT decomposition; see~\cref{sec:preliminaries,sec:NTT} for details. $\vecr=(r_0,r_1,\dots,r_d)$. $\tensB_1=\{\tensX\in\mathbb{C}^{n_1\times n_2\times\cdots\times n_d}:\|\tensX\|_\frob=1\}$.}
    \renewcommand{\arraystretch}{1.2} 
    \setlength{\tabcolsep}{5pt}       
    \begin{tabular}{lcc}
        \toprule
        Property & {Tensor train} & Normalized tensor train \\
        \midrule
        Rank-$\vecr$ approximation & $\proj^{\mathrm{TT-SVD}}_\vecr$ & $\proj_{\tensB_1}^{}\circ\proj^{\mathrm{TT-SVD}}_\vecr$ \\
        \cmidrule(rl){2-2}\cmidrule(rl){3-3}
        Quasi-optimality & $\sqrt{d-1}$ in~\cref{eq:quasi_tt} & $(2\sqrt{d-1}+1)$ in~\cref{eq:quasi}\\
        \midrule
        Parameter space & \multicolumn{2}{c}{$\mathbb{C}^{r_0\times n_1\times r_1}\times\mathbb{C}^{r_1\times n_2\times r_2}\times\cdots\times\mathbb{C}^{r_{d-1}\times n_d\times r_d}$,\quad $r_0=r_d=1$}\\
        \midrule
        Fixed-rank manifold & $\tensM_\vecr$ & $\tensN_\vecr = \tensM_\vecr\cap\tensB_1$\\
        \cmidrule(rl){2-2}\cmidrule(rl){3-3}
        Dimension & $\sum_{k=1}^d r_{k-1}n_kr_k-\sum_{k=1}^{d-1}r_k^2$ & $\sum_{k=1}^d r_{k-1}n_kr_k-\sum_{k=1}^{d-1}r_k^2-1$\\
        \midrule
        \multirow{2}*{Tangent space} & \multicolumn{2}{c}{$\sum_{k=1}^d\llbracket\tensU_1,\dots,\tensU_{k-1},\dot\tensU_k,\tensU_{k+1},\dots,\tensU_d\rrbracket$} \\
        \cmidrule(rl){2-2}\cmidrule(rl){3-3}
        & $\leftunfolding(\dot\tensU_k)^\dagger\leftunfolding(\tensU_k)=0$, $k=1,2,\dots,d-1$ & $\leftunfolding(\dot\tensU_k)^\dagger\leftunfolding(\tensU_k)=0$, $k=1,2,\dots,d$\\
        \bottomrule
    \end{tabular}
    \label{tab:TTvsNTT}
\end{table}

Building upon the developed NTT decomposition, we propose geometric methods for the minimization of smooth functions on $\tensN_\vecr$, where a Riemannian conjugate gradient method, denoted by NTT-RCG, is developed. We demonstrate the effectiveness of NTT-RCG on low-rank tensor recovery, eigenvalue problem, computation of the stabilizer rank in quantum physics, and evaluation of the minimum output R\'enyi $p$-entropy. 

More specifically, we apply the NTT-RCG method for applications in scientific computing. For low-rank tensor recovery, the NTT-RCG method successfully recovers low-rank tensors for both the noiseless and noisy observations. For the eigenvalue problem, we compare the proposed NTT-RCG method with an alternating linear scheme method~\cite{holtz2012alternating}, which is also known as the single-site DMRG in physics. Numerical results suggest that the proposed method performs better than the single-site DMRG with faster convergence and better accuracy on the largest (smallest) eigenvalue. Second, we adopt the NTT-RCG method to the applications in quantum information theory. For the computation of the approximate stabilizer rank, the NTT decomposition not only saves storage but also enables efficient computation of the cost function. Moreover, the NTT-RCG method is indeed able to approximate a non-stabilizer state by several states with much lower non-stabilizerness. Additionally, the NTT-RCG method provides a new practical method for directly evaluating the minimum output R\'enyi $p$-entropy. We consider two typical channels in quantum information: the antisymmetric channel and the generalized amplitude damping channel. The computational cost of the proposed NTT-RCG method scales polynomially with respect to qubits and bond dimensions. We numerically validated that there is no superadditivity up to $12$ qubits with a rank less than $(1,10,10,\dots,10,1)$.

\subsection{Organization}
We introduce the preliminaries of Riemannian geometry and tensor operations in~\cref{sec:preliminaries}. We propose the normalized tensor train decomposition, and develop the Riemannian geometry of the set of fixed-rank tensors in~\cref{sec:NTT}. Section~\ref{sec:applications_sisc} presents applications of the NTT decomposition in scientific computing. Section~\ref{sec:applications_QIT} is devoted to applications in quantum information theory, with problem formulations adopted from the relevant literature. We draw the conclusion in~\cref{sec:conclusion}.

\section{Preliminaries}\label{sec:preliminaries}
In this section, we introduce the preliminaries of Riemannian geometry; see, e.g., \cite{absil2009optimization,boumal2023intromanifolds}. Then, we present notation for tensor operations~\cite{oseledets2011tensor}.

\subsection{Notation for Riemannian geometry}
Assume that a smooth manifold $\tensM$ is embedded in a Euclidean space $\tensE$. The tangent space of $\tensM$ at $x\in\tensM$ is denoted by $\tangent_x\!\tensM$. Let $\tensM$ be endowed with a \emph{Riemannian metric} $g$, where $g_x:\tangent_x\!\tensM\times\tangent_x\!\tensM\to\mathbb{R}$ is a symmetric, bilinear, positive-definite function, and smooth with respect to $x\in\tensM$. The Riemannian metric $g$ introduces a norm $\|\eta\|_x=\sqrt{g_x(\eta,\eta)}$ for $\eta\in\tangent_x\!\tensM$. Given $\bar{\eta}\in\tangent_x\!\tensE\simeq\tensE$, the orthogonal projection operator onto $\tangent_x\!\tensM$ is $\proj_{\tangent_x\!\tensM}(\bar{\eta})$. The \emph{tangent bundle} is denoted by $\tangent\!\tensM=\cup_{x\in\tensM}\tangent_x\!\tensM$. A smooth mapping $\retr:\tangent\!\tensM\to\tensM$ is called a retraction~\cite[Definition 1]{absil2012projection} on $\tensM$ around $x\in\tensM$ if there exists of a neighborhood $U$ of $(x,0)\in\tangent\tensM$ such that 1) $U\subseteq\mathrm{dom}(\retr)$ and $\retr$ is smooth on $U$; 2) $\retr_x(0) = x$ for all $x\in\tensM$; 3) $\mathrm{D}\!\retr_x(\cdot)[0]=\mathrm{id}_{\tangent_x\!\tensM}$. The \emph{vector transport} operator is denoted by $\tensT_{y\gets x}:\tangent_x\!\tensM\to\tangent_y\!\tensM$. The set $\St_{\CC}(p,n)=\{\matx\in\CC^{n\times p}:\matx^\dagger\matx=\matI_p\}$ is referred to as the \emph{complex Stiefel manifold}, where $\matx^\dagger$ is the conjugate transpose of~$\matx$. The complex conjugate of $\matx$ is denoted by $\conj(\matx)=(\conj(x_{i,j}))_{i,j}$. The matrix $\matx^\top$ represents the transpose of $\matx$. 

\subsection{Notation for tensor operations}
The inner product between $\tensX,\tensY\in\CC^{n_1\times\cdots\times n_d}$ is defined by $\langle\tensX,\tensY\rangle := \sum_{i_1=1}^{n_1} \cdots \sum_{i_d=1}^{n_d} \conj(\tensX({i_1,\dots,i_d}))\tensY({i_1,\dots,i_d}).$ The Frobenius norm of a tensor $\tensX$ is defined by $\|\tensX\|_\mathrm{F}:=\sqrt{\langle\tensX,\tensX\rangle}$. The $k$-mode product of a tensor $\tensX$ and a matrix $\mata\in\CC^{n_k\times M}$ is denoted by $\tensX\times_k\mata\in\CC^{n_1\times\cdots\times M\times\cdots\times n_d}$, where the $ (i_1,\dots,i_{k-1},j,i_{k+1},\dots,i_d)$-th entry of $\tensX\times_k\mata$ is $\sum_{i_k=1}^{n_k}x_{i_1\dots i_d}a_{ji_k}$. Given $\vecu_1\in\CC^{n_1}\setminus\{0\},\dots,\vecu_d\in\CC^{n_d}\setminus\{0\}$, a rank-$1$ tensor of size $n_1\times n_2\times\cdots\times n_d$ is defined by the outer product $\tensV=\vecu_1\circ\vecu_2\circ\cdots\circ\vecu_d$, or $v_{i_1,\dots,i_d}=u_{1,i_1}\cdots u_{d,i_d}$ equivalently. The Kronecker product of two matrices $\mata\in\CC^{m_1\times n_1}$ and $\matb\in\CC^{m_2\times n_2}$ is an $(m_1m_2)$-by-$(n_1n_2)$ matrix defined by $\mata\ox\matb:=(a_{ij}\matb)_{ij}$. A tensor $\tensU_k\in\mathbb{C}^{r_{k-1}\times n_k\times r_k}$ can be reshaped to \emph{left} and \emph{right unfoldings} defined by $\leftunfolding(\tensU_k)\in\CC^{(r_{k-1}n_k)\times r_k}$ and $\rightunfolding(\tensU_k)\in\CC^{r_{k-1}\times (n_kr_k)}$. Given a tensor $\tensX\in\CC^{n_1\times n_2\times n_3}$ and $\mata\in\CC^{n_1\times n_1}$, $\matb\in\CC^{n_2\times n_2}$, $\matc\in\CC^{n_3\times n_3}$, we provide the following equalities for the left and right unfoldings,
\begin{equation}
    \label{eq: matricization}
    \begin{aligned}
        \leftunfolding(\tensX\times_1\mata\times_2\matb\times_3\matc)&=(\matb\ox\mata)\leftunfolding(\tensX)\matc^\top,\\
        \rightunfolding(\tensX\times_1\mata\times_2\matb\times_3\matc)&=\mata\rightunfolding(\tensX)(\matc\ox\matb)^\top.
    \end{aligned}
\end{equation}

We introduce the notation for tensors in the tensor train format as follows. Denote the index set $\{1,2,\dots,n\}$ by~$[n]$. The $k$-th unfolding matrix of a tensor $\tensX\in\mathbb{C}^{n_1\times n_2\times\cdots\times n_d}$ is defined by $\matx_{\langle k\rangle}\in\CC^{(n_1n_2\cdots n_k)\times(n_{k+1}n_{k+2}\cdots n_d)}$ for $k\in[d-1]$ with 
\[\matx_{\langle k\rangle}\Big(i_1+\sum_{j=2}^{k}(i_j-1)\prod_{\ell=1}^{j-1}n_\ell,\ i_{k+1}+\sum_{j=k+2}^{d}(i_j-1)\prod_{\ell=k+1}^{j-1}n_\ell\Big)=\tensX(i_1,i_2,\dots,i_d).\]
The TT rank of $\tensX$ is defined by the array
\begin{equation*}
\ranktt(\tensX)\coloneqq (1,\rank(\matx_{\langle 1\rangle}),\rank(\matx_{\langle 2\rangle}),\dots,\rank(\matx_{\langle d-1\rangle}),1).
\end{equation*} 
It follows from~\cite[Theorem~1]{holtz2012manifolds} that for $\tensX$ with $\ranktt(\tensX)=\vecr=(r_0,r_1,\dots,r_d) $ and $r_0=r_d=1$, one can yield the TT decomposition $\tensX=\llbracket\tensU_1,\tensU_2,\dots,\tensU_d\rrbracket$ (or $\tensX(i_1,i_2,\dots,i_d)=\matu_1(i_1)\matu_2(i_2)\cdots\matu_d(i_d)$ for $i_k\in[n_k]$ and $k\in[d]$ equivalently) with core tensors $\tensU_k\in\CC^{r_{k-1}\times n_k\times r_k}$ by using $d$ sequential SVDs. The process is referred to as the TT-SVD algorithm in~\cite[Algorithm~1]{oseledets2011tensor}. More specifically, starting with the first-unfolding matrix $\matx_{\langle1\rangle}$, the TT-SVD algorithm executes by 1) sequentially reshaping the tensor into a matrix; 2) decomposing the matrix by SVD; 3) reshaping the resulting decomposition to yield a core tensor $\tensU_k$ and a smaller tensor; see the TT-SVD algorithm in~\cref{fig:TT-SVD}. The TT-SVD operator $\proj^\mathrm{TTSVD}_\vecr$ satisfies the following quasi-optimality
\begin{equation}
    \label{eq:quasi_tt}
    \|\proj^\mathrm{TTSVD}_\vecr(\tensA)-\tensA\|_\frob\leq\sqrt{d-1}\|\proj_\vecr(\tensA)-\tensA\|_\frob
\end{equation}
for all $\tensA\in\CC^{n_1\times n_2\times\cdots\times n_d}$, where $\proj_\vecr(\tensA)$ is the best rank-$\vecr$ approximation of $\tensA$ in the TT format.

The \emph{interface} matrices $\matx_{\leq k}$ and $\matx_{\geq k+1}$ of~$\tensX$ are defined by $\matx_{\leq k}(i_1+\sum_{j=2}^{k}(i_j-1)\prod_{\ell=1}^{j-1}n_\ell, :)=\matu_1(i_1)\cdots\matu_k(i_k)$ and $\matx_{\geq k+1}(i_{k+1}+\sum_{j=k+2}^{d}(i_j-1)\prod_{\ell=k+1}^{j-1}n_\ell,:) = (\matu_{k+1}(i_{k+1})\cdots\matu_d(i_d))^\top$ respectively. It holds that $\matx_{\langle k\rangle}=\matx_{\leq k}^{}\matx_{\geq k+1}^\top$ and the interface matrices can be constructed recursively by
\begin{equation}\label{eq: recursive interface}
    \matx_{\leq k}=(\matI_{n_k}\ox\matx_{\leq k-1})\leftunfolding(\tensU_k)\quad\text{and}\quad\matx_{\geq k+1}=(\matx_{\geq k+2}\ox\matI_{n_{k+1}})\rightunfolding(\tensU_{k+1})^\top.
\end{equation}
A tensor $\tensX=\llbracket\tensU_1,\tensU_2,\dots,\tensU_d\rrbracket$ is called $k$-orthogonal if $\leftunfolding(\tensU_j)\in\St_{\CC}(r_j,r_{j-1}n_j)$ for $j\in[k-1]$ and $\rightunfolding(\tensU_j)^\top\in\St_{\CC}(r_{j-1},n_jr_j)$ for $j=k+1,k+2,\dots,d$. The tensor is called left- or right-orthogonal if $k=d$ or $k=1$, respectively. It follows from~\cite[Section 3.1]{steinlechner2016riemannian} that any tensor $\tensX$ can be left- or right-orthogonalized via QR decomposition. The set of fixed-rank tensors in the TT format is denoted by
\begin{equation}
    \tensM_\vecr=\left\{\tensX\in\CC^{n_1\times n_2\times\cdots\times n_d}:\ranktt(\tensX)=\vecr
    \right\},
    \label{eq:Mr}
\end{equation}
which is a complex submanifold of $\CC^{n_1\times n_2\times\cdots\times n_d}$; see~\cite[Theorem~14]{haegeman2014geometry}.

\section{Normalized tensor train decomposition}\label{sec:NTT}
In this section, we first define the NTT decomposition. Then, we provide an approximate projection to compress a full tensor into the NTT format via low-rank approximation. Additionally, we delve into the Riemannian geometry of the set of fixed-rank tensors in NTT decomposition.

\subsection{NTT decomposition}
The tensor train decomposition is able to decompose a large tensor into smaller core tensors. However, it does not intrinsically consider the unit-norm constraint, which appears in various applications. To this end, we introduce the normalized tensor train decomposition, which approximates a full tensor by a TT tensor with unit Frobenius norm. 

\begin{definition}[normalized tensor train decomposition]
    Given a tensor $\tensA\in\CC^{n_1\times n_2\times\cdots\times n_d}$, the normalized tensor train decomposition of $\tensA$ aims to approximate the tensor $\tensA$ by a TT tensor $\llbracket\tensU_1,\tensU_2,\dots,\tensU_d\rrbracket$ satisfying
    \[\tensA\approx\llbracket\tensU_1,\tensU_2,\dots,\tensU_d\rrbracket\quad\text{and}\quad\|\llbracket\tensU_1,\tensU_2,\dots,\tensU_d\rrbracket\|_\frob=1.\]
    where $\tensU_k\in\CC^{r_{k-1}\times n_k\times r_k}$ for $k\in[d]$ and $r_k$ is a positive integer for $k\in[d-1]$. The positive integers $r_1,r_2,\dots,r_{d-1}$ are referred to as the \emph{bond dimensions} in physics. 
\end{definition}

Note that NTT decomposition provides a unit-norm low-rank approximation to any tensor instead of an exact decomposition of a tensor. If $\|\tensA\|_\frob=1$, the tensor $\tensA$ admits an \emph{exact} NTT decomposition of the form 
\[\tensA=\llbracket\tensU_1,\tensU_2,\dots,\tensU_d\rrbracket\quad\text{and}\quad\|\llbracket\tensU_1,\tensU_2,\dots,\tensU_d\rrbracket\|_\frob=1,\]
where the first equality holds in NTT decomposition and $r_k = \rank(\mata_{\langle k\rangle})$ for $k\in[d-1]$. In fact, the resulting NTT decomposition is equivalent to the standard TT decomposition of $\tensA$ since the TT decomposition preserves the norm. If {$\|\tensA\|_\frob\neq1$}, the NTT decomposition can be achieved by firstly performing the standard TT decomposition and subsequently normalizing the resulting TT tensor. Moreover, it satisfies the quasi-optimality; see~\cref{prop:quasi}. It is worth noting that NTT decomposition can also be defined for tensors in $\mathbb{R}^{n_1\times n_2\times\cdots\times n_d}$.

\paragraph{Orthogonalization of core tensors}
A tensor in the NTT format can be left- or right-orthogonalized via QR decomposition in a same fashion as the TT format. However, the orthogonality of the resulting core tensors is different between TT and NTT formats. Specifically, given a left-orthogonalized NTT tensor $\tensX=\llbracket\tensU_1,\tensU_2,\dots,\tensU_d\rrbracket\in\tensN_\vecr$, i.e., $\leftunfolding(\tensU_k)^\dagger\leftunfolding(\tensU_k)=\matI_{r_{k}}$ for $k=[d-1]$, it follows from $\|\tensX\|_\frob=1$ and $\leftunfolding(\tensU_d)\in\CC^{r_{d-1}n_d}$ that
\begin{equation}
    \label{eq:left_orth}\leftunfolding(\tensU_d)^\dagger\leftunfolding(\tensU_d)=\|\tensU_d\|_\frob^2=\|\matx_{\leq k-1}^\dagger\rightunfolding(\tensU_d)\|_\frob^2=\|\matx_{\langle d-1\rangle}\|_\frob^2=\|\tensX\|_\frob^2=1
\end{equation}
since $\matx_{\leq k-1}^\dagger\matx_{\leq k-1}^{}=\matI_{r_{d-1}}$. Therefore, in contrast with TT decomposition, where the last core $\tensU_d$ does not satisfy left-orthogonality after left orthogonalization, all core tensors $\tensU_k$, including $\tensU_d$, are left-orthogonal in NTT decomposition.

\subsection{NTT-SVD algorithm}
Given a full tensor $\tensA\in\CC^{n_1\times n_2\times\cdots\times n_d}$ and an array $\vecr=(1,r_1,r_2,\dots,r_{d-1},1)$, a natural question is how to compute the best rank-$\vecr$ approximation of $\tensA$ in the NTT format, i.e., the following metric projection
\[\proj_{\tensN_\vecr}(\tensA):=\argmin_{\tensX\in\tensN_\vecr}\|\tensX-\tensA\|_\frob.\]

Note that the set
\[\tensN_\vecr=\tensM_\vecr\cap \tensB_1=\{\tensX\in\CC^{n_1\times n_2\times\cdots\times n_d}:\ranktt(\tensX)=\vecr\ \text{and}\ \|\tensX\|_\frob=1\}\]
is the intersection of the set $\tensM_\vecr$ in~\cref{eq:Mr} and the sphere $\tensB_1=\{\tensX\in\CC^{n_1\times n_2\times\cdots\times n_d}:\|\tensX\|_\frob=1\}$. Consequently, it is possible to construct the projection onto $\tensN_\vecr=\tensM_\vecr\cap \tensB_1$ via projections onto $\tensM_\vecr$ and $\tensB_1$ sequentially, which are exactly the best rank-$\vecr$ approximation in the TT format and normalization.

It is known that the best rank-$\vecr$ approximation in the TT format does not enjoy a closed-form expression~\cite{oseledets2011tensor}, leading to the downside that the projection $\proj_{\tensN_\vecr}(\tensA)$ is also impractical. Therefore, we consider an approximate projection by firstly implementing TT-SVD on $\tensA$ to yield a rank-$\vecr$ approximation $\proj^{\mathrm{TTSVD}}_\vecr(\tensA)\in\tensM_\vecr$, and normalizing $\proj^{\mathrm{TTSVD}}_\vecr(\tensA)$ onto~$\tensB_1$, i.e.,
\[\proj^{\mathrm{NTTSVD}}_\vecr(\tensA) := \proj_{\tensB_1}(\proj^{\mathrm{TTSVD}}_\vecr(\tensA));\]
see the flowchart in~\cref{fig:TT-SVD}. We refer to the approximate projection $\proj^{\mathrm{NTTSVD}}_\vecr$ as the \emph{NTT-SVD algorithm}. It should be noted that the two operations can not be switched. In practice, the normalization onto $\tensB_1$ can be efficiently computed by normalizing the last core tensor $\tensU_d$ since the core tensors $\tensU_1,\tensU_2,\dots,\tensU_{d-1}$ generated by the TT-SVD algorithm are left-orthogonal and $\|\llbracket\tensU_1,\tensU_2,\dots,\tensU_d\rrbracket\|_\frob=\|\tensU_d\|_\frob$ from~\cref{eq:left_orth}.

\begin{figure}[htbp]
    \centering
    \begin{tikzpicture}[scale=1]
        \footnotesize
        \def\x{2.25};
        \def\y{1.3};
        \def\h{0.3};
        \node (X) at (0,0) {$\tensA$};
        \node (X1) at ($(X)+(\x,0)$) {$\mata_{(1)}\approx\matu_1^{}\ \ \matS_1^{}\matv_1^\top$};
        \draw[->] (X) -- (X1);
        \draw[thick,dashed] ($(X1)+(-0.22,\h)$) rectangle ($(X1)+(0.25,-\h)$); 
        \node[above] at ($0.75*(X)+0.25*(X1)$) {\scriptsize mat};

        \node (U1) at ($(X1)+(0,-\y)$) {
            \begin{tikzpicture}
                \def\l{0.5}
                \node at ($0.5*(\l,\l)$) {$\tensU_1$};
                \draw (\l,\l) -- (\l+0.2,\l+0.15) -- (0.2,\l+0.15) -- (0,\l) -- cycle;
            \end{tikzpicture}
        };
        \draw[thick,->] ($(X1)+(0,-\h)$) -- (U1);

        \draw[thick,dotted] ($(X1)+(0.38,\h)$) rectangle ($(X1)+(1.22,-\h)$); 
        \node (T1) at ($(X1)+1.5*(\x,0)$) {$\matt\approx\matu_2^{}\ \ \matS_2^{}\matv_2^\top$};
        \draw[->] (X1) -- (T1);
        \node[above] at ($0.47*(X1)+0.53*(T1)$) {\scriptsize ten, mat};
        
        \draw[thick,dashed] ($(T1)+(-0.45,\h)$) rectangle ($(T1)+(0.08,-\h)$); 
        \node (U2) at ($(T1)+(-0.15-0.025,-\y)$) {
            \begin{tikzpicture}
                \def\l{0.5}
                \node at ($0.5*(\l,\l)$) {$\tensU_2$};
                \draw (0,0) rectangle (\l,\l);
                \draw (\l,\l) -- (\l+0.2,\l+0.15) -- (0.2,\l+0.15) -- (0,\l);
                \draw (\l,\l) -- (\l+0.2,\l+0.15) -- (\l+0.2,0.15) -- (\l,0);                    
            \end{tikzpicture}
        };
        \draw[->] ($(T1)+(-0.15-0.025,-\h)$) -- (U2);
        
        \draw[thick,dotted] ($(T1)+(0.22,\h)$) rectangle ($(T1)+(1,-\h)$); 
        \node (dots) at ($(T1)+0.8*(\x,0)$) {$\cdots$};
        \draw[->] (T1) -- (dots);

        \node (Td1) at ($(dots)+0.4*(\x,0)$) {$\matT$};

        \node (Td) at ($(Td1)+0.62*(\x,0)$) {$\matT$};

        \draw[->] (dots) -- (Td1);
        \draw[thick,->] (Td1) -- (Td);
        \node (Ud) at ($(Td)+1.1*(\x,0)$) {
            \begin{tikzpicture}
                \def\l{0.5}
                \node at ($0.5*(\l,\l)$) {$\tensU_d$};
                \draw (0,0) rectangle (\l,\l);
            \end{tikzpicture}
        };
        \node (Ud1) at ($(Td)+1.1*(\x,0)+(0,-\y)$) {
            \begin{tikzpicture}
                \def\l{0.5}
                \node at ($0.5*(\l,\l)$) {$\tensU_d$};
                \draw (0,0) rectangle (\l,\l);
            \end{tikzpicture}
        };
        \draw[->] ($(Td)+(0.15,0)$) -- (Ud);
        \draw[->] (Ud) -> (Ud1);
        \node[left] at ($0.5*(Td)+0.5*(0,\h)+0.5*(Ud)+(0.6,0.1)$) {$\matT/\|\matT\|_\frob$};

        \draw[thick, densely dotted, color=gray!80!white,rounded corners=2pt] ($(X)+(-0.3,1)$) rectangle ($(Td1)+(0.4,-0.5)-(0,\y)$); 

        \node[fill=gray!20!white,draw=gray,dashed,thick,inner sep=4pt,rounded corners=2pt] (TTSVD) at ($0.5*(X)+0.5*(Td1)+(0,1.1)$) {\normalsize TT-SVD algorithm $\proj^{\mathrm{TTSVD}}_\vecr$};

        \draw[thick, densely dotted, color=gray!80!white,rounded corners=2pt] ($(Td)+(-0.5,1)$) rectangle ($(Ud)+(0.5,-0.5)-(0,\y)$); 

        \node[fill=gray!20!white,draw=gray,dashed,thick,inner sep=4pt,rounded corners=2pt] (Norm) at ($0.5*(Ud)+0.5*(Td)+(0,1.1)$) {\normalsize Normalization $\proj_{\tensB_1}^{{\color{gray!20!white}T}}$};

        \node[fill=gray!20!white,draw,thick,inner sep=5pt] (approj) at ($0.5*(TTSVD)+0.5*(Norm)+(0,\y)$) {\normalsize NTT-SVD algorithm $\proj^{\mathrm{NTTSVD}}_\vecr$};

        \draw[thick,->] (approj) -- (TTSVD);
        \draw[thick,->] (approj) -- (Norm);
    \end{tikzpicture}
    \caption{Flowchart of the NTT-SVD algorithm. Mat: matricization; ten: tensorization}
    \label{fig:TT-SVD}
\end{figure}
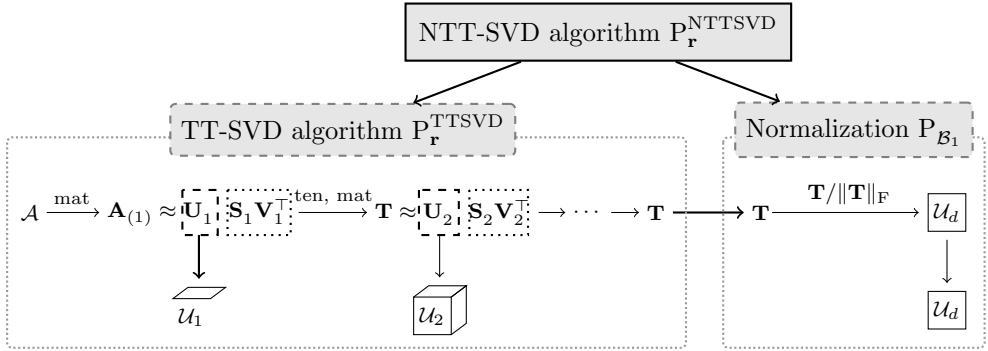

The proposed approximate projection $\proj^{\mathrm{NTTSVD}}_\vecr$ is not guaranteed to be a metric projection $\proj_{\tensN_\vecr}$. Nevertheless, the following proposition illustrates the relationship between $\proj^{\mathrm{NTTSVD}}_\vecr$ and $\proj_{\tensN_\vecr}$, which is referred to as quasi-optimality.
\begin{proposition}[quasi-optimality]\label{prop:quasi}
    The approximate projection satisfies 
    \begin{equation}
        \label{eq:quasi}
        \|\proj_{\tensN_\vecr}(\tensA)-\tensA\|_\frob\leq\|\proj^{\mathrm{NTTSVD}}_\vecr(\tensA)-\tensA\|_\frob\leq (2\sqrt{d-1}+1)\|\proj_{\tensN_\vecr}(\tensA)-\tensA\|_\frob
    \end{equation}
    for any tensor $\tensA\in\CC^{n_1\times n_2\times\cdots\times n_d}$ and rank parameter $\vecr$.
\end{proposition}
\begin{proof}
    It is straightforward to verify that $\|\proj_{\tensN_\vecr}(\tensA)-\tensA\|_\frob\leq\|\proj^{\mathrm{NTTSVD}}_\vecr(\tensA)-\tensA\|_\frob$ since $\proj_{\tensN_\vecr}(\tensA)$ is the best rank-$\vecr$ approximation of $\tensA$ in the NTT format. For the second inequality, it holds that 
    \begin{align}
        \|\proj^{\mathrm{NTTSVD}}_\vecr(\tensA)-\tensA\|_\frob&=\|\proj_{\tensB_1}(\proj_{\tensM_\vecr}^\mathrm{TTSVD}(\tensA))-\tensA\|_\frob\nonumber\\
        &\leq\|\proj_{\tensB_1}(\proj_{\tensM_\vecr}^\mathrm{TTSVD}(\tensA))-\proj_{\tensM_\vecr}^\mathrm{TTSVD}(\tensA)\|_\frob+\|\proj_{\tensM_\vecr}^\mathrm{TTSVD}(\tensA)-\tensA\|_\frob\nonumber\\
        &\leq \|\proj_{\tensN_\vecr}(\tensA)-\proj_{\tensM_\vecr}^\mathrm{TTSVD}(\tensA)\|_\frob+\|\proj_{\tensM_\vecr}^\mathrm{TTSVD}(\tensA)-\tensA\|_\frob\label{eq:proof_1}\\
        &\leq \|\proj_{\tensN_\vecr}(\tensA)\!-\!\tensA\|_\frob+\|\tensA\!-\!\proj_{\tensM_\vecr}^\mathrm{TTSVD}(\tensA)\|_\frob+\|\proj_{\tensM_\vecr}^\mathrm{TTSVD}(\tensA)\!-\!\tensA\|_\frob\nonumber\\
        &\leq \|\proj_{\tensN_\vecr}(\tensA)-\tensA\|_\frob+2\sqrt{d-1}\|\proj_{\tensM_\vecr}(\tensA)-\tensA\|_\frob\label{eq:proof_2}\\
        &\leq (2\sqrt{d-1}+1)\|\proj_{\tensN_\vecr}(\tensA)-\tensA\|_\frob,\label{eq:proof_3}
    \end{align}
    where the inequality~\eqref{eq:proof_1} follows from the metric projection $\proj_{\tensB_1}$ and $\proj_{\tensN_\vecr}(\tensA)\in\tensB_1$, the inequality~\eqref{eq:proof_2} follows from the quasi-optimality~\eqref{eq:quasi_tt} of TT format, and the inequality~\eqref{eq:proof_3} follows from $\tensN_\vecr\subseteq\tensM_\vecr$.
\end{proof}

\subsection{Manifold structure}
In fact, the NTT decomposition generates sets of low-rank tensors with manifold structure. Specifically, given an array of integers $\vecr=(1,r_1,r_2,\dots,r_{d-1},1)$, the set of rank-$\vecr$ tensors in the NTT format
\[\tensN_\vecr=\big\{\tensX\in\CC^{n_1\times n_2\times\cdots\times n_d}:\ranktt(\tensX)=\vecr,~\|\tensX\|_\frob=1\big\}\] 
is a complex submanifold of $\CC^{n_1\times n_2\times\cdots\times n_d}$. Recall that
\[\tensN_\vecr=\tensM_\vecr\cap \tensB_1\]
is the intersection of two manifolds: the manifold of fixed-rank tensors $\tensM_\vecr$ in the TT format~\cref{eq:Mr} and the unit sphere $\tensB_1$. We observe that two manifolds $\tensM_\vecr$ and $\tensB_1$ intersect \emph{transversally}, i.e., $\tangent_\tensX\!\tensM_\vecr+\tangent_\tensX\!\tensB_1=\CC^{n_1\times n_2\times\cdots\times n_d}$ holds for all $\tensX\in\tensN_\vecr$. Therefore, it implies from~\cite[Theorem~6.30]{lee2012smooth} that $\tensN_\vecr$ is a smooth manifold.

Specifically, recall the following parametrization of tangent space of $\tensM_\vecr$ at $\tensX=\llbracket\tensU_1,\tensU_2,\dots,\tensU_d\rrbracket\in\tensM_\vecr$ (see~\cite{holtz2012manifolds,haegeman2014geometry}):
\begin{equation}
    \tangent_\tensX\!\tensM_\vecr=\left\{
        \begin{array}{l}
            ~~~\llbracket\dot{\tensU}_1,\tensU_2,\tensU_3,\dots,\tensU_d\rrbracket\\
            +~\llbracket\tensU_1,\dot{\tensU}_2,\tensU_3,\dots,\tensU_d\rrbracket\\
            ~\vdots\\
            +~\llbracket\tensU_1,\tensU_2,\dots,\tensU_{d-1},\dot{\tensU}_d\rrbracket
        \end{array}\!:\begin{array}{r}
            \dot{\tensU}_k\in\CC^{r_{k-1}\times n_k\times r_k},k\in[d],\\
            \leftunfolding(\dot{\tensU}_k)^\dagger\leftunfolding(\tensU_k)=0,k\in[d-1]
        \end{array}
    \right\}.
    \label{eq:tangent_space_TT}
\end{equation}
The tangent space of $\tensB_1$ at $\tensX\in \tensB_1$ can be represented by 
\begin{equation}
    \label{eq:Tangent_B1}
    \tangent_\tensX\!\tensB_1=\{\tensV\in\CC^{n_1\times n_2\times\cdots\times n_d}\!:\langle\tensV,\tensX\rangle=0\}.
\end{equation}
It follows that $(\tangent_\tensX\!\tensB_1)^\perp=\{t\tensX:t\in\CC\}$. Since we do not impose an orthogonality condition on $\leftunfolding(\dot{\tensU}_d)$ in~\eqref{eq:tangent_space_TT}, it holds that $t\tensX\in\tangent_\tensX\!\tensM_\vecr$ by letting $\dot{\tensU}_k=0$ for $k=1,2,\dots,d-1$ and $\dot{\tensU}_d=t\tensU_d$. Therefore, we conclude that 
\[\CC^{n_1\times n_2\times\cdots\times n_d}=(\tangent_\tensX\!\tensB_1)^\perp+\tangent_\tensX\!\tensB_1\subseteq\tangent_\tensX\!\tensM_\vecr+\tangent_\tensX\!\tensB_1\subseteq\CC^{n_1\times n_2\times\cdots\times n_d},\]
and thus $\tensN_\vecr$ is a smooth manifold.

\subsection{Riemannian geometry of NTT tensors}
We develop the Riemannian geometry of $\tensN_\vecr$, including the tangent space, the Riemannian metric, the projection onto the tangent space, and a retraction; see~\cref{fig:geomtools} for an illustration. Specifically, there are two steps for searching along $\tensN_\vecr$. 1) Projection onto the tangent space: given a point $\tensX\in\tensN_\vecr$ and a direction $\tensA\in\mathbb{C}^{n_1\times n_2\times\cdots\times n_d}$, we project $\tensA$ onto $\tangent_\tensX\!\tensN_\vecr$ and yield $\tensV$, which can be interpreted by successive projections onto $\tangent_\tensX\!\tensM_\vecr$ and $\tangent_\tensX\!\tensB_1$; 2) moving on the manifold: given $s>0$, we retract $(\tensX+s\tensV)$ onto $\tensN_\vecr$ by the NTT-SVD algorithm, where the TT-SVD algorithm and normalization are successively implemented. 

\begin{figure}[htbp]
    \centering
    \includegraphics[width=0.9\textwidth]{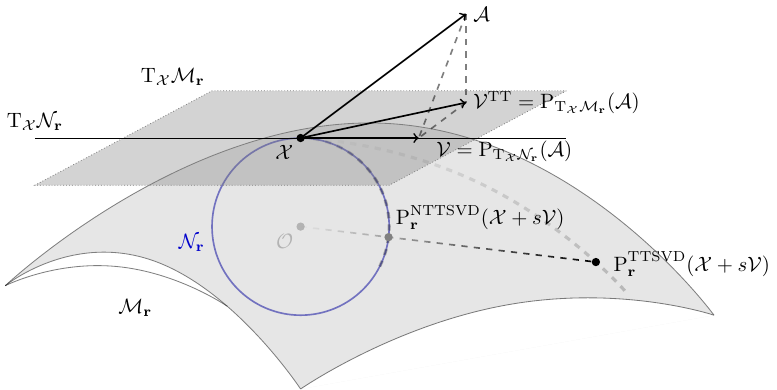}
    \caption{An illustration of the geometry of $\tensN_\vecr$. $\mathcal{O}\in\CC^{n_1\times n_2\times \cdots\times n_d}$: zero tensor.}
    \label{fig:geomtools}
\end{figure}

\paragraph{Tangent space}
Since $\tensM_\vecr$ and $\tensB_1$ intersect transversally, it follows from~\cite{lee2012smooth} that the tangent space of $\tensN_\vecr$ can be represented by the intersection of tangent spaces $\tangent_\tensX\!\tensM_\vecr$ and $\tangent_\tensX\!\tensB_1$, i.e., $\tangent_\tensX\!\tensN_\vecr=\tangent_\tensX\!\tensM_\vecr\cap\tangent_\tensX\!\tensB_1$. Therefore, we yield the following tangent space parametrization.
\begin{proposition}[tangent space]
    Given $\tensX=\llbracket\tensU_1,\tensU_2,\dots,\tensU_d\rrbracket\in\tensN_\vecr$ with left-orthogonal cores, the tangent space of $\tensN_\vecr$ at $\tensX$ can be parametrized by
    \begin{equation}
    \tangent_\tensX\!\tensN_\vecr=\left\{
        \begin{array}{l}
            ~~~\llbracket\dot{\tensU}_1,\tensU_2,\tensU_3,\dots,\tensU_d\rrbracket\\
            +~\llbracket\tensU_1,\dot{\tensU}_2,\tensU_3,\dots,\tensU_d\rrbracket\\
            ~\vdots\\
            +~\llbracket\tensU_1,\tensU_2,\dots,\tensU_{d-1},\dot{\tensU}_d\rrbracket
        \end{array}\!:\!
        \begin{array}{r}
            \dot{\tensU}_k\in\CC^{r_{k-1}\times n_k\times r_k},k\in[d],\\\leftunfolding(\dot{\tensU}_k)^\dagger\leftunfolding(\tensU_k)=0,k\in[d-1],\\
            \langle\dot{\tensU}_d,\tensU_d\rangle=0
        \end{array}
    \right\}.
    \label{eq:tangent_space_unit_TT}
\end{equation}
\end{proposition}
\begin{proof}
    Denote the right-hand side by $T$. On the one hand, for any vector $\tensV\in T$, it holds that $\tensV\in\tangent_\tensX\!\tensM_\vecr$ and 
    \begin{align*}
        \langle\tensV,\tensX\rangle&=\langle\llbracket\dot{\tensU}_1,\tensU_2,\dots,\tensU_d\rrbracket,\llbracket\tensU_1,\tensU_2,\dots,\tensU_d\rrbracket\rangle+\langle\llbracket\tensU_1,\dot{\tensU}_2,\tensU_3,\dots,\tensU_d\rrbracket,\llbracket\tensU_1,\tensU_2,\dots,\tensU_d\rrbracket\rangle\\
            &~~~~+\cdots+\langle\llbracket\tensU_1,\tensU_2,\dots,\tensU_{d-1},\dot{\tensU}_d\rrbracket,\llbracket\tensU_1,\tensU_2,\dots,\tensU_d\rrbracket\rangle\\
            &=\langle\llbracket\tensU_1,\tensU_2,\dots,\tensU_{d-1},\dot{\tensU}_d\rrbracket,\llbracket\tensU_1,\tensU_2,\dots,\tensU_d\rrbracket\rangle\\
            &=\langle\dot{\tensU}_d,\tensU_d\rangle=0,
    \end{align*}
    where the equalities follow from the orthogonality conditions $\leftunfolding(\dot{\tensU}_k)^\dagger\leftunfolding(\tensU_k)=0$ for $k\in[d-1]$ and $\langle\dot{\tensU}_d,\tensU_d\rangle=0$. Therefore, we have $\tensV\in\tangent_\tensX\!\tensB_1$ from~\cref{eq:Tangent_B1} and thus $T\subseteq\tangent_\tensX\!\tensM_\vecr\cap\tangent_\tensX\!\tensB_1=\tangent_\tensX\!\tensN_\vecr$.

    On the other hand, for a tangent vector $\tensV\in\tangent_\tensX\!\tensN_\vecr\subseteq\tangent_\tensX\!\tensM_\vecr$, there exists $\dot\tensU_k\in\CC^{r_{k-1}\times n_k\times r_k}$ such that $\tensV=\llbracket\dot{\tensU}_1,\tensU_2,\tensU_3,\dots,\tensU_d\rrbracket+\llbracket\tensU_1,\dot{\tensU}_2,\tensU_3,\dots,\tensU_d\rrbracket+\cdots+\llbracket\tensU_1,\tensU_2,\dots,\tensU_{d-1},\dot{\tensU}_d\rrbracket$
    and $\leftunfolding(\dot{\tensU}_k)^\dagger\leftunfolding(\tensU_k)=0$ for $k\in[d-1]$. Since $\tensV\in\tangent_\tensX\!\tensN_\vecr\subseteq\tangent_\tensX\!\tensB_1$, it holds that $\langle\tensV,\tensX\rangle=0$, i.e., $\langle\dot{\tensU}_d,\tensU_d\rangle=0$. Therefore, $\tensV\in T$. 
    
    Consequently, it holds that $T=\tangent_\tensX\!\tensN_\vecr$.
\end{proof}

Note that $\langle\dot{\tensU}_d,\tensU_d\rangle=0$ is equivalent to $\leftunfolding(\dot{\tensU}_d)^\dagger\leftunfolding(\tensU_d)=0$. Therefore, the parameter $\dot\tensU_d$ in~\cref{eq:tangent_space_unit_TT} satisfies the orthogonality condition while $\dot\tensU_d$ in~\cref{eq:tangent_space_TT} is arbitrary. In practice, it suffices to store the parameters $\dot{\tensU}_1,\dot{\tensU}_2,\dots,\dot{\tensU}_d$ for a tangent vector.

\paragraph{Projection onto the tangent space}
Subsequently, we compute the projection of a tensor onto the tangent space. We adopt the inner product $\langle\cdot,\cdot\rangle$ as the Riemannian metric of $\tensN_\vecr$.
\begin{proposition}\label{prop:proj_tangent}
    Given $\tensX=\llbracket\tensU_1,\tensU_2,\dots,\tensU_d\rrbracket\in\tensN_\vecr$ with left-orthogonal cores, the projection of $\tensA\in\CC^{n_1\times n_2\times\dots\times n_d}$ onto the tangent space $\tangent_\tensX\!\tensN_\vecr$ can be given by a tangent vector $\proj_{\tangent_\tensX\!\tensN_\vecr}(\tensA)\in\tangent_\tensX\!\tensN_\vecr$ with parameters $\tensW_k\in\CC^{r_{k-1}\times n_k\times r_k}$ satisfying
    \begin{equation*}
        \begin{aligned}
            \leftunfolding(\tensW_k) &= (\matI_{r_{k-1}n_k}-\matP_k)(\matI_{n_k}\ox\matx_{\leq k-1})^\dagger\mata_{\langle k\rangle}\conj(\matx_{\geq k+1})(\matx_{\geq k+1}^\top\conj(\matx_{\geq k+1}))^{-1}\\
            \rmvec(\tensW_d) &= (\matI_{r_{d-1}n_d}-\matP_d)(\matI_{n_d}\ox\matx_{\leq d-1})^\dagger\rmvec(\mata)
        \end{aligned}
    \end{equation*}
    for $k\in[d-1]$, where $\matP_k=\leftunfolding(\tensU_k)\leftunfolding(\tensU_k)^\dagger$ is the orthogonal projection operator onto the range of $\leftunfolding(\tensU_k)$.
\end{proposition}
\begin{proof}
    Since $\proj_{\tangent_\tensX\!\tensN_\vecr}(\tensA)\in\tangent_\tensX\!\tensN_\vecr$, it can be parametrized by~\eqref{eq:tangent_space_unit_TT} with parameters $\tensW_k\in\CC^{r_{k-1}\times n_k\times r_k}$ for $k=1,2,\dots,d$. We aim to figure out the parameters.

    Recall the parametrization~\eqref{eq:tangent_space_unit_TT} and denote each summand by $\tensV_k$, i.e., $\tensV=\tensV_1+\tensV_2+\cdots+\tensV_d$. We observe that $\langle\tensA-\proj_{\tangent_\tensX\!\tensN_\vecr}(\tensA),\tensV\rangle=0$ holds for any tangent vector $\tensV\in\tangent_\tensX\!\tensN_\vecr$. Then, we obtain that
    \begin{align*}
        &~~~~\langle\tensA-\proj_{\tangent_\tensX\!\tensN_\vecr}(\tensA),\tensV_k\rangle= \langle\tensA-\proj_{\tangent_\tensX\!\tensN_\vecr}(\tensA),\llbracket\tensU_1,\dots,\tensU_{k-1},\dot\tensU_k,\tensU_{k+1},\dots,\tensU_d\rrbracket\rangle\\
        &=\langle\mata_{\langle k\rangle}-(\matI_{n_k}\ox\matx_{\leq k-1})\leftunfolding(\tensW_k)\matx_{\geq k+1}^\top,(\matI_{n_k}\ox\matx_{\leq k-1})\leftunfolding(\dot\tensU_k)\matx_{\geq k+1}^\top\rangle\\
        &=\langle(\matI_{n_k}\ox\matx_{\leq k-1})^\dagger\mata_{\langle k\rangle}\conj(\matx_{\geq k+1})-\leftunfolding(\tensW_k)\matx_{\geq k+1}^\top\conj(\matx_{\geq k+1}),\leftunfolding(\dot\tensU_k)\rangle\\
        &=0
    \end{align*}
    holds for any $\leftunfolding(\dot\tensU_k)^\dagger\leftunfolding(\tensU_k)=0$ for $k\in[d]$,
    where we use the orthogonality of interface matrices~\eqref{eq: recursive interface} and the facts that $\matx_{\leq k-1}=(\tensV_k)_{\leq k-1}$ and $\matx_{\geq k+1} = (\tensV_k)_{\geq k+1}$. Therefore, we conclude that 
    \[\leftunfolding(\tensW_k) = (\matI_{r_{k-1}n_k}-\matP_k)(\matI_{n_k}\ox\matx_{\leq k-1})^\dagger\mata_{\langle k\rangle}\conj(\matx_{\geq k+1})(\matx_{\geq k+1}^\top\conj(\matx_{\geq k+1}))^{-1}\]
    for $k=1,2,\dots,d-1$ and 
    \[\rmvec(\tensW_d) = (\matI_{r_{d-1}n_d}-\matP_d)(\matI_{n_d}\ox\matx_{\leq d-1})^\dagger\rmvec(\mata).\]
\end{proof}

It follows from the proof of Proposition~\ref{prop:proj_tangent} that the projection onto the tangent space $\tangent_\tensX\!\tensN_\vecr$ can be expressed by the projection onto the tangent space $\tangent_\tensX\!\tensM_\vecr$ and onto the tangent space $\proj_{\tangent_\tensX\!\tensB_1}$ of the unit sphere. 
\begin{corollary}
    It holds that
    \[\proj_{\tangent_\tensX\!\tensN_\vecr} = \proj_{\tangent_\tensX\!\tensB_1}\circ\proj_{\tangent_\tensX\!\tensM_\vecr}=\proj_{\tangent_\tensX\!\tensM_\vecr}\circ\proj_{\tangent_\tensX\!\tensB_1}.\]
\end{corollary}

In practice, we observe that the parameter $\tensW_k$ involves $(\matx_{\geq k+1}^\top\conj(\matx_{\geq k+1}))^{-1}$, which can be ill-conditioned. Therefore, inspired by~\cite[\S 3.3]{steinlechner2016riemannian}, we consider the representation $\tensX=\llbracket\tensU_1,\dots,\tensU_{k-1},\tilde\tensU_k,\tensY_{k+1},\dots,\tensY_d\rrbracket$ with $k$-orthogonal parameters, i.e., $\tensU_1,\dots,\tensU_{k-1}$ are left-orthogonal, $\tensY_{k+1},\dots,\tensY_d$ are right-orthogonal, and $\tilde\tensU_k$ is not guaranteed to be left- or right-orthogonal. Subsequently, each summand $\tensV_k$ can be $k$-orthogonalized to $\tensV_k=\llbracket\tensU_1,\dots,\tensU_{k-1},\tilde\tensW_k,\tensY_{k+1},\dots,\tensY_d\rrbracket$. Consequently, we obtain an equivalent representation of $\proj_{\tangent_\tensX\!\tensN_\vecr}(\tensA)\in\tangent_\tensX\!\tensN_\vecr$:
\begin{equation}
    \label{eq:proj}
        \proj_{\tangent_\tensX\!\tensN_\vecr}(\tensA)=\sum_{k=1}^d\llbracket\tensU_1,\dots,\tensU_{k-1},\tilde\tensW_k,\tensY_{k+1},\dots,\tensY_d\rrbracket
\end{equation}
with $\leftunfolding(\tilde{\tensW_k}) = (\matI_{r_{k-1}n_k}-\matP_k)(\matI_{n_k}\ox\matx_{\leq k-1})^\dagger\mata_{\langle k\rangle}\conj(\matx_{\geq k+1})$. Note that $\matx_{\geq k+1}\in\St_{\CC}(r_k,n_{k+1}n_{k+1}\cdots n_d)$ due to the $k$-orthogonality.

\paragraph{Retraction}
For navigating on the manifold $\tensN_\vecr$, a \emph{retraction} mapping is required. Recall that a mapping $\retr:\tangent\!\tensN_\vecr\to\tensN_\vecr$ is called a retraction~\cite[Definition 1]{absil2012projection} on $\tensN_\vecr$ around $\tensX\in\tensN_\vecr$ if there exists a neighborhood $U$ of $(\tensX,0)\in\tangent\!\tensN_\vecr$ such that 1) $U\subseteq\mathrm{dom}(\retr)$ and $\retr$ is smooth on $U$; 2) $\retr_\tensX(0) = \tensX$ for all $\tensX\in\tensN_\vecr$; 3) $\mathrm{D}\!\retr_\tensX(\cdot)[0]=\mathrm{id}_{\tangent_\tensX\!\tensN_\vecr}$, the identity mapping on the tangent space. 
\begin{proposition}
    Given $\tensX\in\tensN_\vecr$ and a tangent vector $\tensV\in\tangent_\tensX\!\tensN_\vecr$, the mapping $\retr_\tensX(\tensV)=\proj^{\mathrm{NTTSVD}}_\vecr(\tensX+\tensV)$ defines a retraction.
\end{proposition}
\begin{proof}
    It suffices to prove the three aforementioned properties. For the first property, it follows from~\cite[Proposition 4]{steinlechner2016riemannian} that there exists a neighborhood $U\subseteq\CC^{n_1\times n_2\times\cdots\times n_d}$ of $\tensX$ and such that $0\notin U$ and $\proj^\mathrm{TTSVD}_\vecr$ are smooth. Since $\proj_{\tensB_1}$ is also smooth on $U$ and $\retr_\tensX(\tensV)=\proj_{\tensB_1}(\proj^\mathrm{TTSVD}_\vecr(\tensX+\tensV))$, the mapping $\retr$ is smooth in a neighborhood of $(\tensX,0)\in\tensN_\vecr\times\tangent_\tensX\!\tensN_\vecr$. The second property is straightforward. 

    We prove the third property through quasi-optimality~\eqref{eq:quasi}. Since the metric projection $\proj_{\tensN_\vecr}$ is a retraction on $\tensN_\vecr$, we have 
    \[\|(\tensX+t\tensV)-\retr_\tensX(t\tensV)\|_\frob\leq(2\sqrt{d-1}+1)\|(\tensX+t\tensV)-\proj_{\tensN_\vecr}(\tensX+t\tensV)\|_\frob=\mathcal{O}(t^2).\]
    Therefore, it holds that $\retr_\tensX(t\tensV)=\tensX+t\tensV+\mathcal{O}(t^2)$, i.e., $\mathrm{D}\!\retr_\tensX(\cdot)[0]=\mathrm{id}_{\tangent_\tensX\!\tensN_\vecr}$. Consequently, $\retr$ defines a retraction on $\tensN_\vecr$ around $\tensX\in\tensN_\vecr$.
\end{proof}

\subsection{Geometric methods}
Based on the geometry of the set $\tensN_\vecr$ of fixed-rank tensors in the NTT format, we consider the following optimization problem on the smooth manifold $\tensN_\vecr$,
\begin{equation}
    \label{eq:general_optim}        \min~f(\tensX),\quad\subjectto~\tensX\in\tensN_\vecr=\tensM_\vecr\cap\tensB_1,
\end{equation}
where $f:\CC^{n_1\times n_2\times\cdots\times n_d}$ is a smooth function.

\begin{algorithm}[htbp]
    \caption{Riemannian conjugate gradient method for~\cref{eq:general_optim} (NTT-RCG)}
    \label{alg:NTT-RCG}
    \begin{algorithmic}[1]
        \REQUIRE Initial guess $\tensX^{(0)}\in\tensN_\vecr$, $t=0$, $\beta^{(0)}=0$
        \WHILE{the stopping criteria are not satisfied}
            \STATE Compute the parameters $\tilde\tensW_k^{(t)}$ of the tangent vector \\
            $\tensV^{(t)}=\proj_{\tangent_{\tensX^{(t)}}\!\tensM}(-\nabla f(\tensX^{(t)}))+\beta^{(t)}\tensT_{\tensX^{(t)}\gets\tensX^{(t-1)}}\tensV^{(t-1)}$ by~\eqref{eq:proj}.
            \STATE Choose stepsize $s^{(t)}$.
            \STATE Update $\tensX^{(t+1)}=\proj^{\mathrm{NTTSVD}}_\vecr(\tensX^{(t)}+s^{(t)}\tensV^{(t)})$ by~\cref{fig:TT-SVD} and $t=t+1$.
        \ENDWHILE
        \ENSURE $\tensX^{(t)}$.
    \end{algorithmic}
\end{algorithm}

We adopt the Riemannian conjugate gradient method to solve~\eqref{eq:general_optim}; see Algorithm~\ref{alg:NTT-RCG}. 
We set the vector transport $\tensT_{\tensX^{(t)}\gets\tensX^{(t-1)}}\tensV^{(t-1)}$ as the orthogonal projection~\eqref{eq:proj}. Therefore, the parameters $\tilde\tensW_k^{(t)}$ of $\tensV^{(t)}$ can be computed by adding the parameters of $\proj_{\tangent_{\tensX^{(t)}}\!\tensM}(-\nabla f(\tensX^{(t)}))$ and $\beta^{(t)}\proj_{\tangent_{\tensX^{(t)}}\!\tensM}(\tensV^{(t-1)})$ obtained by~\eqref{eq:proj}. Given the representation of $\tensV=\proj_{\tangent_\tensX\!\tensN_\vecr}(-\nabla f(\tensX))$ in~\eqref{eq:proj}, the tensor $\tensX+\tensV$ can be represented by a tensor in the TT format where the $(i_1,i_2,\dots,i_d)$-th element is given by
\begin{align*}
    \begin{bmatrix}
        \tilde\matW_1(i_1) & \matu_1(i_1)
    \end{bmatrix}&
    \begin{bmatrix}
        \matY_2(i_2) & 0\\
        \tilde\matW_2(i_2)& \matu_2(i_2)
    \end{bmatrix}
    \begin{bmatrix}
        \matY_3(i_3) & 0\\
        \tilde\matW_3(i_3)& \matu_3(i_3)
    \end{bmatrix}\\
    &\cdots
    \begin{bmatrix}
        \matY_{d-1}(i_{d-1}) & 0\\
        \tilde\matW_{d-1}(i_{d-1})& \matu_{d-1}(i_{d-1})
    \end{bmatrix}
    \begin{bmatrix}
        \matY_d(i_d) \\
        \matu_d(i_d) + \tilde\matW_d(i_d)
    \end{bmatrix},
\end{align*}
where $\tensX=\llbracket\tensU_1,\tensU_2,\dots,\tensU_d\rrbracket$ and $\tensX=\llbracket\tensY_1,\tensY_2,\dots,\tensY_d\rrbracket$ are two equivalent NTT decompositions of $\tensX$ with the left- and right-orthogonal core tensors, respectively. Subsequently, the NTT-SVD algorithm can be efficiently implemented.

\begin{remark}\label{rem:FD}
In practice, if $\nabla f(\tensX)$ is a sparse tensor or can be represented by a low-rank TT tensor, the projected gradient $\proj_{\tangent_\tensX\!\tensN_\vecr}(\nabla f(\tensX))$ can be efficiently computed in a similar fashion as~\cite[\S 4.2]{steinlechner2016riemannian}. For instance, the Euclidean gradient of the objective function $f$ in eigenvalue problems can be represented by a tensor in the TT format, and thus the projected gradient can be efficiently computed; see~\cref{subsec:eig_prob} for details. For applications in quantum information theory, the objective function $f$ can be efficiently computed while the Euclidean gradient $\nabla f$ is a full tensor. Therefore, we adopt a finite-difference approach to approximate the projected gradient: 1) generate orthogonal bases $\tensV_1,\tensV_2,\dots,\tensV_{\dim(\tensN_\vecr)}$ of $\tangent_\tensX\!\tensN_\vecr$; 2) approximate the projected gradient by
\[\proj_{\tangent_\tensX\!\tensN_\vecr}(\nabla f(\tensX))\approx\sum_{k=1}^{\dim(\tensN_\vecr)}\alpha_k\tensV_k\quad\text{with}\quad\alpha_k=\frac{f(\proj^{\mathrm{NTTSVD}}_\vecr(\tensX+t\tensV_k))-f(\tensX)}{t}\]
for sufficiently small $t$.
\end{remark}

\section{Applications in scientific computing}\label{sec:applications_sisc} 
In this section, we apply the NTT decomposition to two applications in scientific computing: recovery of low-rank tensors and high-dimensional eigenvalue problems.

We introduce the default settings. In general, the tensor train-related computations are based on the TTeMPS toolbox\footnote{Available at \url{https://www.epfl.ch/labs/anchp/index-html/software/ttemps/}.}, and the proposed NTT-RCG method is implemented in the Manopt toolbox\footnote{Available at \url{https://www.manopt.org/}} v7.1.0, a Matlab library for geometric methods. All experiments are performed on a workstation with two Intel(R) Xeon(R) Processors Gold 6330 (at 2.00GHz$\times$28, 42M Cache) and 512GB of RAM running Matlab R2019b under Ubuntu 22.04.3. The codes of the proposed methods are available at \url{https://github.com/JimmyPeng1998}.

\subsection{Low-rank tensor recovery}\label{subsec:lrtc}
Given a partially observed tensor $\cA\in\tensN_\vecr$ on an index set $\Omega\in[n_1]\times[n_2]\times\cdots\times[n_d]$, we aim to recover the tensor $\tensA$ from its entries on $\Omega$ by solving the following optimization problem
\begin{equation}
    \begin{aligned}
        \min\quad & f(\tensX)=\frac12\|\proj_{\Omega}(\tensX)-\proj_\Omega(\tensA)\|_\frob^2\\
        \subjectto\quad & \tensX\in\tensN_\vecr,
    \end{aligned}
\end{equation}
where $\proj_\Omega$ is defined by $\proj_\Omega(\tensA)(i_1,i_2,\dots,i_d)=\tensA(i_1,i_2,\dots,i_d)$ if $(i_1,i_2,\dots,i_d)\in\Omega$, otherwise $\proj_\Omega(\tensA)(i_1,i_2,\dots,i_d)=0$. The numerical performance of NTT-RCG is measured by the training error $\|\proj_{\Omega}(\tensX)-\proj_\Omega(\tensA)\|_\frob/\|\proj_\Omega(\tensA)\|_\frob$ and test error $\|\proj_{\Gamma}(\tensX)-\proj_\Gamma(\tensA)\|_\frob/\|\proj_\Gamma(\tensA)\|_\frob$ for another validation set $\Gamma\in[n_1]\times[n_2]\times\cdots\times[n_d]$. 

\paragraph{Test on noiseless data}
We consider the noiseless case, i.e., $\tensA\in\tensN_\vecr$ is exactly a low-rank tensor. We aim to show the ability of the NTT-RCG method in recovering a low-rank tensor under different tensor sizes $n$ and sample sizes $|\Omega|$. Following the settings in~\cite[\S 5.3]{steinlechner2016riemannian}, we set $d=5$, $\vecr=(1,3,3,3,3,1)$, tensor size $n=50,100,\dots,400$ and sample size $2000,4000,\dots,60000$. For each combination of $(n,|\Omega|)$, we run the NTT-RCG method five times. We call a successful recovery by the NTT-RCG method if the test error achieves less than $10^{-4}$ within 250 iterations. Figure~\ref{fig:LRTC_results} (left) reports the phase plot for the NTT-RCG method, where the white block indicates successful recovery in all five runs, the black block indicates failure of recovery in all five runs, and the red line represents $\mathcal{O}(n\log(n))$. The phase plot suggests similar scaling behavior to existing results; see, e.g., \cite[\S 5.3]{steinlechner2016riemannian}.

\paragraph{Test on noisy data}
We consider the noisy case, i.e., $\tensA=\hat{\tensA}+\lambda\tensE/\|\tensE\|_\frob$ consists of a unit-norm low-rank tensor $\hat{\tensA}\in\tensN_\vecr$ and noise tensor $\tensE\in\mathbb{R}^{n_1\times n_2\times\cdots\times n_d}$ with noise level $\lambda$. Each element of $\tensE$ is sampled i.i.d. from the standard Gaussian distribution $N(0,1)$. We set $\lambda=10^{-4},10^{-6},\dots,10^{-12},0$, $d=3$, $n=100$, $\vecr=(1,r_1,r_2,1)=(1,3,3,1)$, and $|\Omega|=10dnr_1^2$. Figure~\ref{fig:LRTC_results} (right) shows the convergence results for the NTT-RCG method. We observe that the NTT-RCG method successfully recovers the underlying low-rank tensor under different noise levels.

\begin{figure}[htbp]
    \centering
    \includegraphics[width=0.465\textwidth]{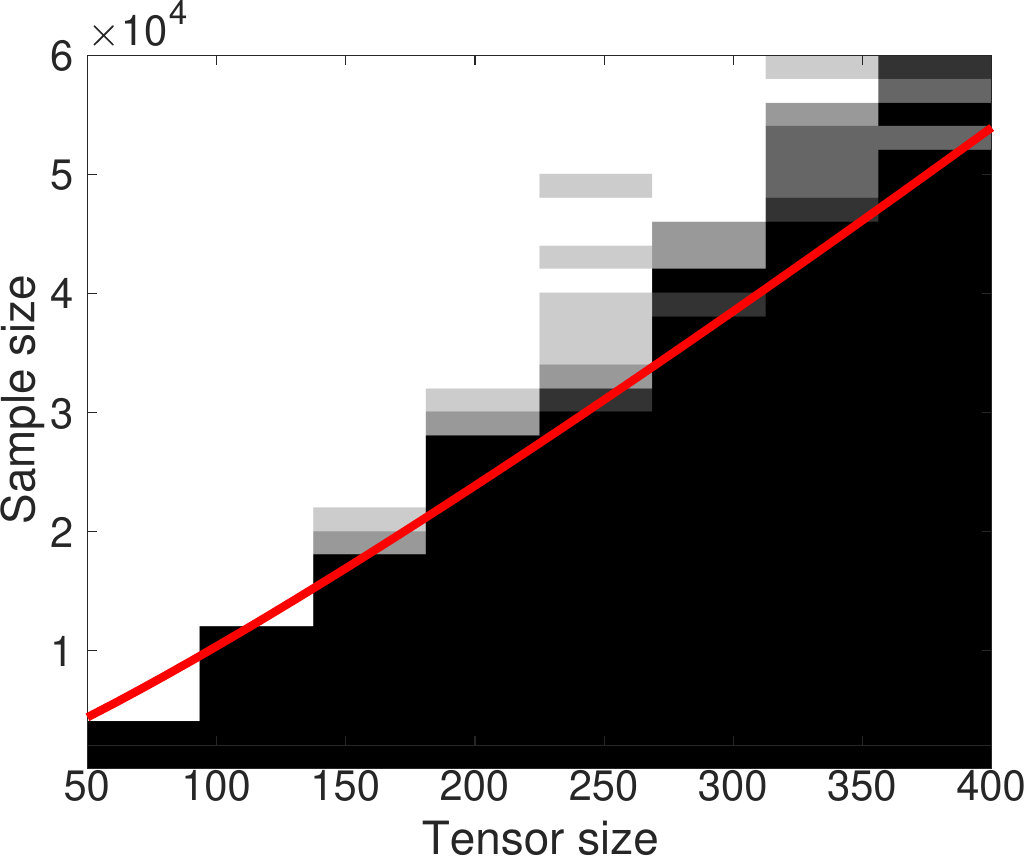}
    \includegraphics[width=0.48\textwidth]{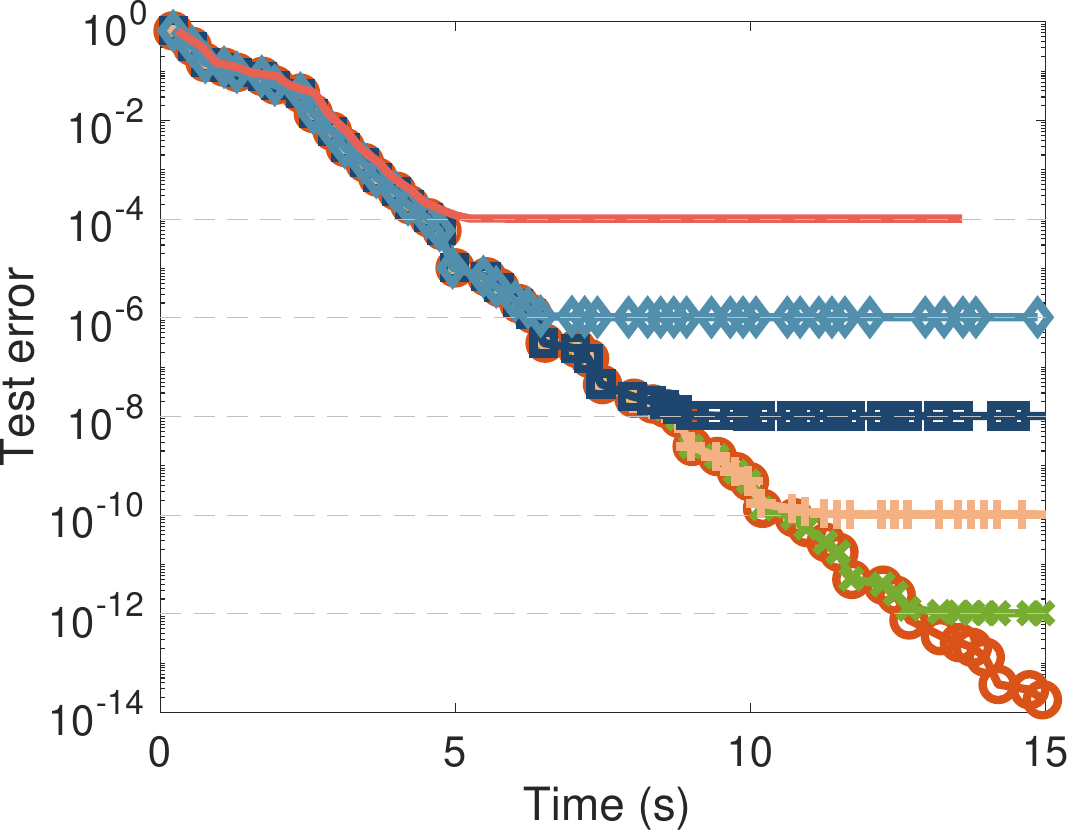}
    \caption{Left: phase plot of recovery results for five runs. The white block indicates successful recovery in all five runs, while the black block indicates failure of recovery in all five runs. Right: test error under noise levels $\lambda=10^{-4},10^{-6},\dots,10^{-12},0$.}
    \label{fig:LRTC_results}
\end{figure}

\subsection{Eigenvalue problem with tensor product structure}\label{subsec:eig_prob}
We consider the following tensor product-structured smallest (largest) eigenvalue problem 
\begin{equation}\label{eq:eigproblem}
    \begin{aligned}
        \min_{\vecx}(\max_{\vecx})\quad & f(\vecx) = \vecx^\top\matH \vecx= \vecx^\top\left(\sum_{\ell=1}^L\matH_{\ell,d}\ox \matH_{\ell,d-1}\ox\cdots\ox\matH_{\ell,1}\right)\vecx\\
        \subjectto\quad & \vecx\in\mathbb{R}^{n_1n_2\cdots n_d},\ \|\vecx\|_2^2 = 1,
    \end{aligned}
\end{equation}
i.e., the matrix $\matH$ can be represented by the sum of Kronecker products of matrices $\matH_{\ell,k}\in\mathbb{R}^{n_k\times n_k}$ for $\ell\in[L]$, $k\in[d]$ and some integer $n_k$. A typical example of such a matrix is a discretization of the $d$-dimensional Laplace operator of the form
\begin{equation}
    \label{eq:Laplace}
    \matH=\matT_{n_d}\ox\matI_{n_{d-1}}\ox\cdots\ox\matI_{n_1}+\cdots+\matI_{n_d}\ox\matI_{n_{d-1}}\ox\cdots\ox\matI_{n_{d-1}}\ox\matT_{n_1},
\end{equation}
where $\matT_n=\mathrm{tridiag}(-1,2,-1)\in\mathbb{R}^{n\times n}$ is a tridiagonal matrix. 

Solving the problem~\eqref{eq:eigproblem} directly is computationally intractable since the number of parameters of $\vecx$ in~\eqref{eq:eigproblem} scales exponentially with respect to $d$. To this end, in light of the tensor product structure, we resort to the normalized tensor train decomposition and restrict~\eqref{eq:eigproblem} to the subset $\tensN_\vecr\subseteq\mathbb{R}^{n_1\times n_2\times\cdots\times n_d}$, i.e.,
\begin{equation}\label{eq:eigproblem_Mr}
    \begin{aligned}
        \min_{\tensX}(\max_{\tensX}) &\qquad  f(\tensX) = \rmvec(\tensX)^\top\matH\rmvec(\tensX)\\
        \subjectto\quad  &\qquad \tensX=\llbracket\tensU_1,\tensU_2,\dots,\tensU_d\rrbracket\in\tensN_\vecr.
    \end{aligned}
\end{equation}
It is worth noting that the computational cost of the objective function $f$ can be significantly reduced by using~\cref{lem:tensor_product}.

\begin{proposition}\label{lem:tensor_product}
    Given $\tensX=\llbracket\tensU_1,\tensU_2,\dots,\tensU_d\rrbracket$ and $\matK_k\in\mathbb{R}^{n_k\times n_k}$, it holds that
    \[(\matK_d\ox\matK_{d-1}\ox\cdots\ox\matK_1)\mathrm{vec}(\tensX)=\mathrm{vec}(\llbracket\tensU_1\times_2\matK_1,\tensU_2\times_2\matK_2,\dots,\tensU_d\times_2\matK_d\rrbracket).\]
\end{proposition}
\begin{proof}
    We start from the first unfolding matrix $\matx_{\langle1\rangle}$ and yield  
    \begin{align*}
        (\matK_d\ox\cdots\ox\matK_1)\mathrm{vec}(\tensX)
        &=(\matK_d\ox\matK_{d-1}\ox\cdots\ox\matK_1)\mathrm{vec}(\leftunfolding(\tensU_1)\matx_{\geq 2}^\top)\nonumber\\
        &=(\matK_d\ox\matK_{d-1}\ox\cdots\ox\matK_1)(\matx_{\geq 2}\ox\matI_{n_1})\mathrm{vec}(\tensU_1)\nonumber\\
        &=((\matK_d\ox\matK_{d-1}\ox\cdots\ox\matK_2)\matx_{\geq 2}\ox\matK_1)\mathrm{vec}(\tensU_1)\nonumber\\
        &=((\matK_d\ox\matK_{d-1}\ox\cdots\ox\matK_2)\matx_{\geq 2}\ox\matI_{n_1})(\matI_{r_1}\ox\matK_1)\mathrm{vec}(\tensU_1)\\
        &=(((\matK_d\ox\matK_{d-1}\ox\cdots\ox\matK_2)\matx_{\geq 2})\ox\matI_{n_1})\mathrm{vec}(\tensU_1\times_2\matK_1),
    \end{align*}
    where the last equality follows from~\eqref{eq: matricization} and $(\matI_{r_1}\ox\matK_1)\mathrm{vec}(\tensU_1) = (\matI_{r_1}\ox\matK_1\ox\matI_{r_0})\mathrm{vec}(\tensU_1) = \mathrm{vec}(\tensU_1\times_2\matK_1)$. 
    Subsequently, by using~\eqref{eq: recursive interface} and~\eqref{eq: matricization}, we obtain that
    \begin{align*}
        (\matK_d\ox\cdots\ox\matK_2)\matx_{\geq 2}
        &= (\matK_d\ox\matK_{d-1}\ox\cdots\ox\matK_2)(\matx_{\geq 3}\ox\matI_{n_2})\rightunfolding(\tensU_2)^\top\\
        &=(((\matK_d\ox\matK_{d-1}\ox\cdots\ox\matK_3)\matx_{\geq 3})\ox\matI_{n_3})(\matI_{r_2}\ox\matK_2)\rightunfolding(\tensU_2)^\top\\
        &=(((\matK_d\ox\matK_{d-1}\ox\cdots\ox\matK_3)\matx_{\geq 3})\ox\matI_{n_3})\rightunfolding(\tensU_2\times_2\matK_2)^\top.
    \end{align*}
    We observe that the tensor product $(\matK_d\ox\matK_{d-1}\ox\cdots\ox\matK_k)\matx_{\geq k}$ can be recursively computed for $k=3,4,\dots,d$ in a same fashion. Consequently, we yield
    \begin{align*}
        &~~~~(\matK_d\ox\matK_{d-1}\ox\cdots\ox\matK_1)\mathrm{vec}(\tensX)\\
        &=\left((((\rightunfolding(\tensU_d\times_2\matK_d)^\top\ox\matI_{d-1})\rightunfolding(\tensU_{d-1}\times_2\matK_{d-1}))\ox\matI_{d-2})\cdots\right)\rmvec(\tensU_1)\\
        &=\mathrm{vec}(\llbracket\tensU_1\times_2\matK_1,\tensU_2\times_2\matK_2,\dots,\tensU_d\times_2\matK_d\rrbracket).
    \end{align*}
\end{proof}

As a result, the cost function $f$ in~\cref{eq:eigproblem_Mr} can be efficiently evaluated via the core tensors $\tensU_1,\tensU_2,\dots,\tensU_d$ of $\tensX$ by
\[f(\tensX) = \rmvec(\tensX)^\top\matH\rmvec(\tensX)= \sum_{\ell=1}^L\langle\tensX,\llbracket\tensU_1\times_2\matH_{\ell,1},\tensU_2\times_2\matH_{\ell,2},\dots,\tensU_d\times_2\matH_{\ell,d}\rrbracket\rangle.\]
In contrast with straightforwardly computing $\rmvec(\tensX)^\top\matH\rmvec(\tensX)$ with $\mathcal{O}(Ln^{2d})$ flops, the computational cost of the new approach is $\mathcal{O}(Ldn^2r_{\max}^2)$, which scales linearly with~$d$, where $r_{\max}=\max\{r_1,r_2,\dots,r_{d-1}\}$.

We apply the NTT-RCG method to the problem~\eqref{eq:eigproblem_Mr}, and the performance is measured by 1) relative error on the smallest (largest) eigenvalue $|\lambda_{\min}-\lambda|/|\lambda_{\min}|$ ($|\lambda_{\max}-\lambda|/|\lambda_{\max}|$); 2) subspace distance $\mathrm{dist}(\tensX,\vecx^*)=\|\rmvec(\tensX)\rmvec(\tensX)^\top-\vecx^*(\vecx^*)^\top\|_\frob$ if available.

\paragraph{Test on Laplace operator}
We consider the discretization of the $d$-dimensional Laplace operator~\cref{eq:Laplace}. The eigenvalue $\lambda_{i_d,i_{d-1},\dots,i_1}$ and the corresponding eigenvector $\vecv_{i_d,i_{d-1},\dots,i_1}$ enjoys a closed-from expression
\begin{equation}
    \label{eq:eigs_laplace}
    \lambda_{i_d,\dots,i_1}=4\sum_{k=1}^d\sin^2(\frac{i_k\pi}{2(n_k+1)})\quad\text{and}\quad\vecv_{i_d,\dots,i_1}(j_d,\dots,j_1)=\prod_{k=1}^d\sin(\frac{i_kj_k\pi}{n_k+1})
\end{equation}
for $i_k,j_k\in[n_k]$ and $k\in[d]$. It follows from~\cref{lem:tensor_product} that the Euclidean gradient of the objective function $f$ at $\tensX=\llbracket\tensU_1,\tensU_2,\dots,\tensU_d\rrbracket$ can be efficiently computed by
\begin{align*}
    &~~~~\nabla f(\tensX)\\
    &=\llbracket\tensU_1\times_2\matt_{n_1},\tensU_2,\dots,\tensU_d\rrbracket+\llbracket\tensU_1,\tensU_2\times_2\matt_{n_2},\dots,\tensU_d\rrbracket+\cdots+\llbracket\tensU_1,\tensU_2,\dots,\tensU_d\times_d\matt_{n_d}\rrbracket\\
    &=\llbracket\tensG_1,\tensG_2,\dots,\tensG_d\rrbracket,
\end{align*}
where
\[\matG_1(i_1)=\begin{bmatrix}
    \tilde{\matu}_1(i_1) & \matu_1(i_1)
\end{bmatrix},\quad\matG_k(i_k)=\begin{bmatrix}
    \matu_k(i_k) & 0 \\
    \tilde{\matu}_k(i_k) & \matu_k(i_k)
\end{bmatrix},\quad\matG_d(i_d)=\begin{bmatrix}
    \matu_d(i_d)\\
    \tilde{\matu}_d(i_d)
\end{bmatrix}\]
for $k=2,3,\dots,d-1$ and $\tilde{\tensU}_k=\tensU_k\times_2\matT_{n_k}$. Note that the tensor size of $\tensG_k$ is independent of $d$, enabling scalable computations.

We compare the proposed NTT-RCG method with the alternating linear scheme method~\cite{holtz2012alternating}, which is also known as the single-site density matrix renormalization group (DMRG)~\cite{White2005,Ulrich2011}. We observe that any eigenvector $\vecv_{i_d,i_{d-1},\dots,i_1}$ in~\cref{eq:eigs_laplace} can be reshaped into a rank-1 tensor in $\mathbb{R}^{n_d\times n_{d-1}\times\cdots\times n_1}$, and thus the eigenvalue problem~\eqref{eq:eigproblem} is equivalent to~\eqref{eq:eigproblem_Mr} with $\vecr=(1,1,\dots,1)$. We set $n_1=n_2=\cdots=n_d=10$ and $d=8,16,32,\dots,256$. \cref{tab:test_on_Laplace,fig:test_on_Laplace} report the numerical results. We observe that all methods converge to the largest eigenvalue and the corresponding eigenvector for different settings of $d$. Notably, the proposed method performs better than the single-site DMRG with faster convergence and better accuracy on the largest eigenvalue. 

\begin{figure}
    \centering
    \includegraphics[width=\linewidth]{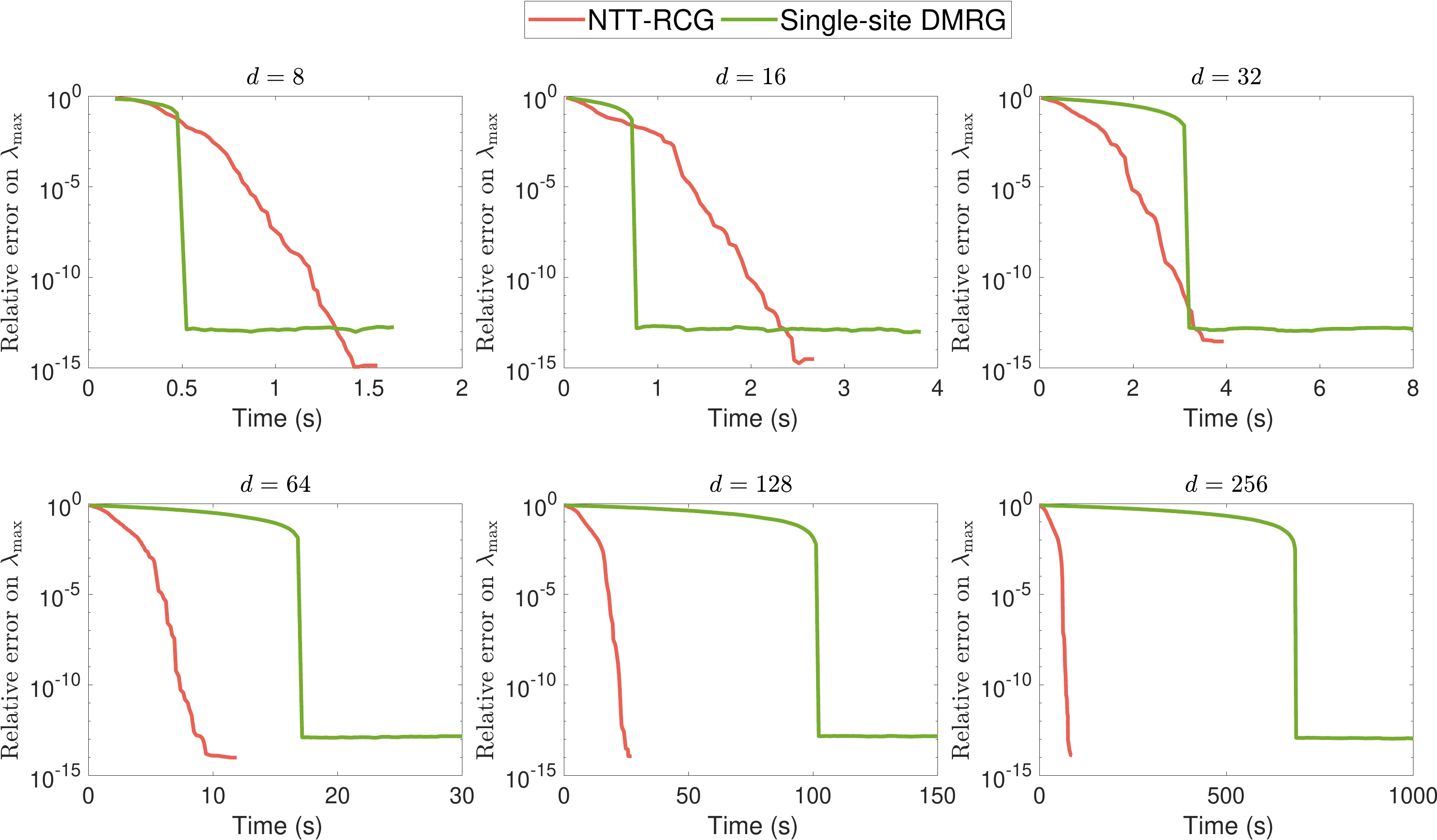}
    \caption{Convergence of two methods for $d=8,16,32,\dots,256$.}
    \label{fig:test_on_Laplace}
\end{figure}

\begin{table}[htbp]
    \centering
    \caption{Numerical results on the discretization of the Laplace operator. Relerr on $\lambda_{\max}$: relative error on the largest eigenvalue.}\footnotesize
    \begin{tabular}{rrrrrrr}
				\toprule
				\multirow{2}*{$d$} & \multicolumn{3}{c}{NTT-RCG} & \multicolumn{3}{c}{Single-site DMRG}\\
				\cmidrule(lr){2-4}\cmidrule(lr){5-7}
				& Time(s) & Relerr on $\lambda_{\max}$ & $\mathrm{dist}(\tensX,\vecx^*)$ & Time(s) & Relerr on $\lambda_{\max}$ & $\mathrm{dist}(\tensX,\vecx^*)$\\
				\midrule
				8 & 1.55 & 1.3598e-15 & 2.1491e-07 & 1.63 & 1.8085e-13 & 5.5952e-06\\
				16 & 2.67 & 3.0596e-15 & 6.3571e-07 & 3.82 & 9.6773e-14 & 5.0423e-06\\
				32 & 3.94 & 2.8896e-14 & 6.1342e-06 & 17.05 & 1.2057e-13 & 9.1063e-06\\
				64 & 11.88 & 9.7453e-15 & 6.1905e-06 & 85.50 & 1.4165e-13 & 1.4724e-05\\
				128 & 27.03 & 1.1558e-14 & 1.2112e-05 & 512.65 & 1.1332e-13 & 1.7337e-05\\
				256 & 86.49 & 1.3258e-14 & 2.4897e-05 & 3438.71 & 1.1751e-13 & 2.6551e-05\\
				\bottomrule
			\end{tabular}
    \label{tab:test_on_Laplace}
\end{table}

\paragraph{Test on transverse field Ising Hamiltonian}
Consider a $d$-site Ising model with a Hamiltonian
\[\matH = -\sum_{k=1}^{d-1}\sigma_k^{z}\sigma_{k+1}^{z}-t\sum_{k=1}^d\sigma_{k}^{x},\]
where $\sigma^x,\sigma^z$ are Pauli matrices defined in~\cref{sec:applications_QIT}, $\sigma_k^{z}=\matI_{2^{k-1}}\ox\sigma^{z}\ox\matI_{2^{d-k}}$, $\sigma^{x}=\matI_{2^{k-1}}\ox\sigma^{x}\ox\matI_{2^{d-k}}$, and $t\in\mathbb{R}$. The eigenvalues of $\matH$ can be efficiently computed via the Jordan--Wigner transformation. However, the eigenvectors are not of closed form. To address this, we consider seeking a low-rank solution of the eigenvector corresponding to the smallest eigenvalue by~\eqref{eq:eigproblem_Mr}. The Euclidean gradient of $f$ at $\tensX\in\tensN_\vecr$ can be represented by the tensor 
$\nabla f(\tensX)=\llbracket\tensG_1,\tensG_2,\dots,\tensG_d\rrbracket$ with
\[\matG_1(i_1)=\left[\begin{smallmatrix}
    \tilde{\matu}_1(i_1) & \breve{\matu}_1(i_1) & \matu_1(i_1)
\end{smallmatrix}\right],\ \matG_k(i_k)=\left[\begin{smallmatrix}
    \matu_k(i_k) & 0 & 0\\
    \breve{\matu}_k(i_k) & 0 & 0\\
    \tilde{\matu}_k(i_k) & \breve{\matu}_k(i_k) & \matu_k(i_k)
\end{smallmatrix}\right],\ \matG_d(i_d)=\left[\begin{smallmatrix}
    \matu_d(i_d)\\
    \breve{\matu}_d(i_d)\\
    \tilde{\matu}_d(i_d)
\end{smallmatrix}\right]\]
for $k=2,3,\dots,d-1$, where $\breve{\tensU}_k = \tensU_k\times_2\matS^{(z)}$ and $\tilde{\tensU}_k=\tensU_k\times_2\matS^{(x)}$ by following~\cref{lem:tensor_product}. 

We set $t=1$, $d=8, 16, 32,\dots,256$. For the sake of better numerical performance, we adopt a rank-increasing strategy to NTT-RCG; see, e.g., \cite[\S4.9]{steinlechner2016riemannian}. Starting from the initial rank $\vecr^{(0)}=(1,1,\dots,1)$, we run NTT-RCG for 50 iterations at each rank, and increase the rank to $\vecr^{(t+1)}=\min\{(1,2,\dots,2^{\lfloor d/2\rfloor},2^{\lfloor d/2\rfloor-1},\dots,1), \vecr^{(t)}+1\}$, until the prescribed maximum rank $\vecr$ is reached. We set the maximum rank $\vecr=\min\{(1,2,\dots,2^{\lfloor d/2\rfloor},2^{\lfloor d/2\rfloor-1},\dots,1), (1,r,r,\dots,1)\}$ with $r=1,4,6,8,\dots,14$. \cref{fig:test_on_Ising,tab:test_on_Ising} report the numerical performance of the NTT-RCG method. First, for small system sizes ($d=8,16$), where the reference eigenvector can be computed by the MATLAB function \texttt{eigs}, the relative error on~$\lambda_{\min}$ and the subspace distance decrease as the parameter~$r$ increases, i.e., a higher-rank solution approximates the eigenvector more accurately. Second, the NTT-RCG method achieves small relative errors on $\lambda_{\min}$ with small rank parameters and acceptable computation time among all choices of $d$. Third, the number of parameters in the NTT representation grows only linearly with~$d$, in sharp contrast to the exponential growth of the full tensor representation. This low-rank structure enables computations on large-scale spin chains (up to $d=256$ sites in our experiments).

\begin{figure}[htbp]
    \centering
    \includegraphics[width=\linewidth]{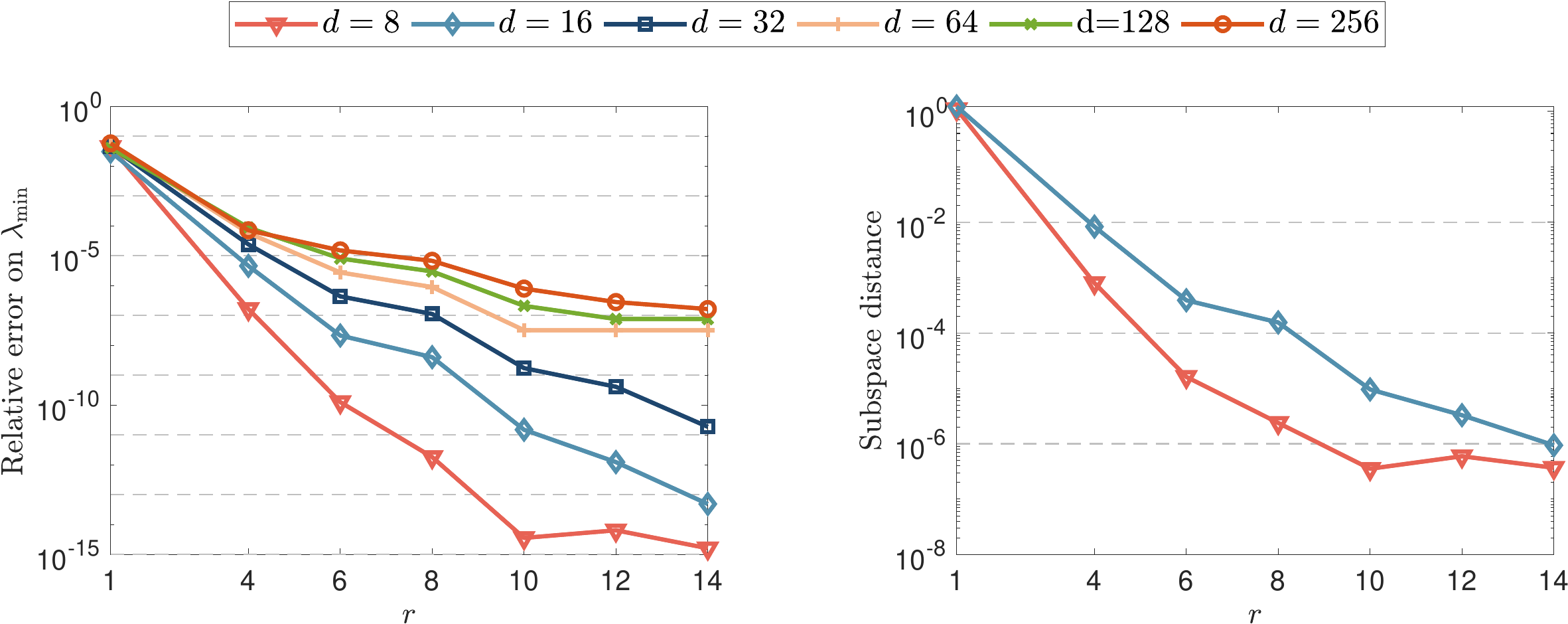}
    \caption{Numerical results on the Ising Hamiltonian. Left: relative error on $\lambda_{\min}$ with $d=8,16,32,\dots,256$. Right: subspace distance with $d=8,16$. }
    \label{fig:test_on_Ising}
\end{figure}

\begin{table}[htbp]
    \centering
    \caption{The performance of the proposed NTT-RCG method on the Ising Hamiltonian. \#params.: number of parameters of $\tensX$. Relerr on $\lambda_{\min}$: relative error on the smallest eigenvalue.}\scriptsize
    \resizebox{1\textwidth}{!}{
    \begin{tabular}{crrcccrrcc}
        \toprule
        $d$ & $r$ & \#params. & Time(s) & Relerr on $\lambda_{\min}$ & $d$ & $r$ & \#params. & Time(s) & Relerr on $\lambda_{\min}$\\
        \cmidrule(lr){1-5}\cmidrule(lr){6-10}
        \multirow{7}*{8} & 1 & 16 & 0.66 & 4.4265e-02 & \multirow{7}*{64} & 1 & 128 & 30.10 & 4.2145e-02\\
         & 4 & 168 & 1.17 & 1.6557e-07 &  & 4 & 1960 & 24.74 & 1.3958e-04\\
         & 6 & 280 & 1.10 & 1.2708e-10 &  & 6 & 4312 & 40.19 & 2.7756e-06\\
         & 8 & 424 & 1.25 & 1.8347e-12 &  & 8 & 7592 & 55.97 & 9.0321e-07\\
         & 10 & 488 & 1.39 & 1.0834e-15 &  & 10 & 11688 & 73.31 & 3.2456e-08\\
         & 12 & 552 & 1.49 & 8.3058e-15 &  & 12 & 16680 & 74.70 & 3.2456e-08\\
         & 14 & 616 & 1.55 & 3.4307e-15 &  & 14 & 22568 & 76.08 & 3.2456e-08\\
        \cmidrule(lr){1-5}\cmidrule(lr){6-10}
        \multirow{7}*{16} & 1 & 32 & 2.24 & 7.3024e-02 & \multirow{7}*{128} & 1 & 256 & 70.28 & 7.8158e-02\\
         & 4 & 424 & 2.86 & 4.6116e-06 &  & 4 & 4008 & 73.67 & 1.6158e-04\\
         & 6 & 856 & 2.63 & 2.1532e-08 &  & 6 & 8920 & 111.12 & 1.1990e-05\\
         & 8 & 1448 & 3.90 & 4.0551e-09 &  & 8 & 15784 & 148.26 & 3.0342e-06\\
         & 10 & 2088 & 4.01 & 1.5077e-11 &  & 10 & 24488 & 184.17 & 2.1188e-07\\
         & 12 & 2856 & 3.47 & 1.2472e-12 &  & 12 & 35112 & 187.36 & 2.1188e-07\\
         & 14 & 3752 & 5.06 & 6.5316e-14 &  & 14 & 47656 & 209.00 & 3.8778e-08\\
        \cmidrule(lr){1-5}\cmidrule(lr){6-10}
        \multirow{7}*{32} & 1 & 64 & 8.92 & 4.5299e-02 & \multirow{7}*{256} & 1 & 512 & 187.28 & 8.2331e-02\\
         & 4 & 936 & 7.65 & 2.4219e-05 &  & 4 & 8104 & 194.31 & 1.8673e-03\\
         & 6 & 2008 & 8.99 & 4.4011e-07 &  & 6 & 18136 & 265.97 & 3.3424e-04\\
         & 8 & 3496 & 9.23 & 1.1340e-07 &  & 8 & 32168 & 331.03 & 5.6170e-05\\
         & 10 & 5288 & 11.32 & 1.7451e-09 &  & 10 & 50088 & 434.01 & 1.2384e-06\\
         & 12 & 7464 & 13.79 & 4.1455e-10 &  & 12 & 71976 & 537.43 & 2.8209e-07\\
         & 14 & 10024 & 13.20 & 1.8888e-11 &  & 14 & 97832 & 594.19 & 1.6480e-07\\
        \bottomrule
    \end{tabular}
    }
    \label{tab:test_on_Ising}
\end{table}


\section{Applications in quantum information theory}\label{sec:applications_QIT} 
In this section, we present applications of the NTT decomposition in computations of stabilizer rank and minimum output entropy of quantum channels, both of which are essential quantities in quantum information theory.

We first introduce notation for quantum information theory. A quantum system $A$ of $n$ qubits is described by the Hilbert space $\cH_A = (\CC^2)^{\ox n}$, with dimension $2^n$. A pure state is a unit-norm vector $\ket{\psi} \in \cH_A$, i.e., $\braket{\psi}{\psi}=1$. A mixed state is described by a density matrix $\rho_A$, which is a positive semidefinite operator on $\cH_A$ with $\tr\rho_A = 1$. A quantum channel $\cN_{A\to B}: \cL(\cH_A) \to \cL(\cH_B)$ is a linear map between spaces of linear operators that is completely positive and trace-preserving. The action of a quantum channel can be expressed via Kraus operators as $\cN(\rho) = \sum_k \matK_k^{} \rho \matK_k^\dagger$ satisfying $\sum_k \matK_k^\dagger \matK_k^{} = \mathbb{I}$. The $n$-qubit Pauli group is defined by $\hat{\cP}_n \coloneqq \{i^{k}\sigma_{h_1}\ox\sigma_{h_2}\ox\cdots\ox\sigma_{h_n} :~k,h_j\in\{ 0,1,2,3\}\}$, where 
\begin{equation*}
\sigma_0 = \left[\begin{array}{cc}
    1 & 0 \\
    0 & 1
\end{array}\right],
~\sigma_1 = \sigma^x = \left[\begin{array}{cc}
    0 & 1 \\
    1 & 0
\end{array}\right],
~\sigma_2 = \sigma^y = \left[\begin{array}{cc}
    0 & -i \\
    i & 0
\end{array}\right],
~\sigma_3 = \sigma^z = \left[\begin{array}{cc}
    1 & 0 \\
    0 & -1
\end{array}\right]
\end{equation*}
are Pauli matrices. The $n$-qubit Pauli group modulo phases is denoted by $\cP_n \coloneqq \hat{\cP}_n/\langle \pm i\idop_{2^n}\rangle$. The $n$-qubit Clifford group is the normalizer of the $n$-qubit Pauli group, defined by $\Cl \coloneqq \big\{U: UPU^\dag \in \hat{\cP}_n, \forall P\in \hat{\cP}_n \big\}.$ The $n$-qubit pure stabilizer states are the orbit of the Clifford group, defined by $\STAB \coloneqq \left\{U\ket{0}^{\ox n}: U\in \Cl\right\}.$ We set the rank parameter 
\begin{equation}
    \label{eq:rank_params}
    \vecr=\min\{(1,2,4,\dots,2^{\lfloor d/2\rfloor},2^{\lfloor d/2\rfloor-1},\dots,2,1), (1,r,r,\dots,r,1)\}
\end{equation}
with an integer $r$ by default. 

\subsection{Approximation of the stabilizer rank}\label{subsec:stab_rank}
The nonstabilizerness, or the magic~\cite{Bravyi_2005}, is a resource essential to unlock the full power of quantum computation and perform universal quantum algorithms that outperform classical counterparts~\cite{Shor_1997,grover1996fast}. This is due to the Gottesman--Knill theorem~\cite{gottesman1997stabilizer}, stating that any \textit{stabilizer states} and \textit{Clifford operations} can be efficiently simulated on classical computers. The \textit{stabilizer rank} provides a quantitative measure of non-stabilizerness~\cite{Bravyi2016,Bravyi2016i}. For an $n$-qubit pure quantum state, it is defined by
\begin{equation}\label{Eq:stabrank}
    \chi(\ket{\psi}) \coloneqq \min\!\big\{ \!R\in\NN_+\!:\!\exists c_1,\dots,c_R\in\CC, \ket{s_1},\dots,\ket{s_R}\in\STAB~\subjectto~\ket{\psi} = \sum_{j=1}^R c_j\ket{s_j} \big\}.
\end{equation}
Determining the stabilizer rank of a quantum state is fundamental to quantifying the magic resources needed for quantum speedups and understanding the boundary between classical and quantum computational power~\cite{Bravyi2016,Bravyi2016i,Bravyi2019simulationofquantum}.

From the definition of the stabilizer rank, it is intractable to directly compute $\chi(\ket{\psi}^{\ox n})$ for tensor powers of a state, because the number of stabilizer states grows as $2^{(1/2+o(1))n^2}$; see~\cite[Proposition 2]{Scott2004}. Therefore, to get an estimation of it, we develop an efficient geometric method to verify whether any given $R$ is a feasible solution for~\eqref{Eq:stabrank}. To this end, we first need an efficient characterization of the stabilizer states. A recently developed magic measure for an $n$-qubit pure quantum state is the $\alpha$-stabilizer R\'enyi entropy (SRE)~\cite{Leone2022}, defined by
\begin{equation*}
    M_{\alpha}(\ket{\psi})\coloneqq\frac{1}{1-\alpha}\log_2 \sum_{P\in\cP_n}\Xi_{P}^{\alpha}(\ket{\psi})-n,
\end{equation*}
where $\Xi_P(\ket{\psi})\coloneqq \bra{\psi} P \ket{\psi}^2 / 2^n$ and $P\in\cP_n$ is an $n$-bit Pauli string. It holds that $M_{\alpha}(\ket{\psi}) = 0$ if and only if $\ket{\psi}\in\STAB$. More crucially, the SREs can be efficiently computed for MPSs. Haug and Piroli~\cite{Haug_2023} developed an approach that can compute the SRE of an MPS with $n$ bond dimensions $r,r,\dots,r$ in terms of the norm of an MPS with $n$ bond dimensions $r^{2\alpha},r^{2\alpha},\dots,r^{2\alpha}$. This motivates us to consider the following optimization problem for a given pure state $\ket{\psi}$, rank parameter~$\vecr$, and number of components $R \in \mathbb{N}_+$:
\begin{equation}\label{eq:stabrank_opt}
\begin{aligned}
    \min_{\{c_j\}_j,\{\ket{\phi_j}\}_j} &\; f(\{c_j\}_j,\{\ket{\phi_j}\}_j)=\frac12\Big\|\sum_{j=1}^{R} c_j \ket{\phi_j} - \ket{\psi}\Big\|_\frob^2 + \lambda\sum_{j=1}^R M_2(\ket{\phi_j}) \\
    \subjectto&\;\; c_1,c_2,\dots,c_R \in \CC,~ \text{each}~ \ket{\phi_j}\in\tensN_{\vecr},
\end{aligned}
\end{equation}
where $\lambda>0$ is a penalty parameter. The objective function consists of a fidelity term $\|\sum_{j=1}^{R} c_j \ket{\phi_j} - \ket{\psi}\|_\frob^2/2$ measuring the reconstruction error and SRE regularization terms $M_2(\ket{\phi_j})$ that promote solutions with low magic. The problem~\eqref{eq:stabrank_opt} can be deemed as an optimization on the product manifold
\[\tensM=\CC^n\times\tensN_{\vecr}\times\tensN_{\vecr}\times\cdots\times\tensN_{\vecr}.\]
We refer to~\cite{gao2025optimization} for optimization on product manifolds.
\begin{remark}
We make several remarks on the proposed method. The first is its effectiveness. If an optimal solution with an objective value of zero is found, we arrive at a valid upper bound on the stabilizer rank of $\ket{\psi}$, i.e., $\chi(\ket{\psi})\leq R$. The second is its efficiency. For the important case of finding the stabilizer rank of a tensor power state $\chi(\ket{\psi}^{\ox n})$, the objective function in~\eqref{eq:stabrank_opt} can be evaluated in time that scales polynomially with $n$, as the target state $\ket{\psi}^{\ox n}$ echoes a rank-1 tensor.
\end{remark}

Finding an exact decomposition with zero objective value provides a strict upper bound on $\chi(|\psi\rangle)$. However, it is numerically challenging. Instead, we propose to seek an $(\epsilon, \delta)$-approximate stabilizer rank, which is defined by the minimum $R$ such that there exists a set $\{c_j, \ket{\phi_j}_{j=1}^R\}$ satisfying two conditions: i) small infidelity, i.e., $1 - |\sum_{j=1}^R c_j \langle \phi_j | \psi \rangle|^2 \le \epsilon$; ii) low magic, i.e., $M_2(\ket{\phi_j}) \le \delta$ for all $j=1, 2,\dots, R$. Finding such a decomposition for a small $R$ provides a physically meaningful approximation of a magic state with states that are \textit{close} to the stabilizer set.

A representative magic state is the qubit $H$ state $\ket{H}=\cos(\pi/8)\ket{0}+\sin(\pi/8)\ket{1}$. Bravyi et al.~\cite{Bravyi2016} showed that the stabilizer rank of $\ket{H^{\ox n}}$ has an upper bound $\chi(\ket{H^{\ox n}})\leq 7^{n/6}$. Therefore, we set $\lambda=1$, $n=2,3,4,5,6$, $R =1,2,\dots,\lceil 7^{n/6}\rceil$ for each $n$, and $\vecr$ with parameter $r=1,2$ and $n\leq 4$ or $\vecr$ with $r=1,2,\dots,5$ in~\cref{eq:rank_params} and $n=5,6$. Since the Euclidean gradient of $M_2(\ket{\psi})$ involves~$4^n$ summands, the Euclidean and projected gradients of the objective function are computationally intractable. Therefore, we adopt the finite difference method (see~\cref{rem:FD}) to approximate the projected gradient by computing directional derivatives along bases of the tangent space of $\tensM$, which involves $\mathcal{O}(2Rnr_{\max}^2)$ evaluations of $f$ with $r_{\max}=\max\{r_1,r_2,\dots,r_{n-1}\}$. 

Numerical results are reported in~\cref{tab:stabrank,fig:stabrank}. We observe that as the number of bases $R$ increases, the infidelity of the approximation decreases. Simultaneously, the maximum SRE among the components remains bounded at a low value. For instance, as shown in~\cref{tab:stabrank}, for $n=4$ qubits, we achieve an infidelity below $1.2\times 10^{-3}$ with $R=3$ components, where each component has an SRE less than $8.3\times 10^{-3}$. It indicates that $R=3$ is a $(1.2\times 10^{-3}, 8.3\times 10^{-3})$-approximate stabilizer rank for $\ket{H^{\ox 4}}$. These results show that the proposed method is able to guide the design of classical simulation algorithms based on the idea of approximately decomposing states into low-magic components.

\begin{figure}[htbp]
    \centering
    \includegraphics[width=\textwidth]{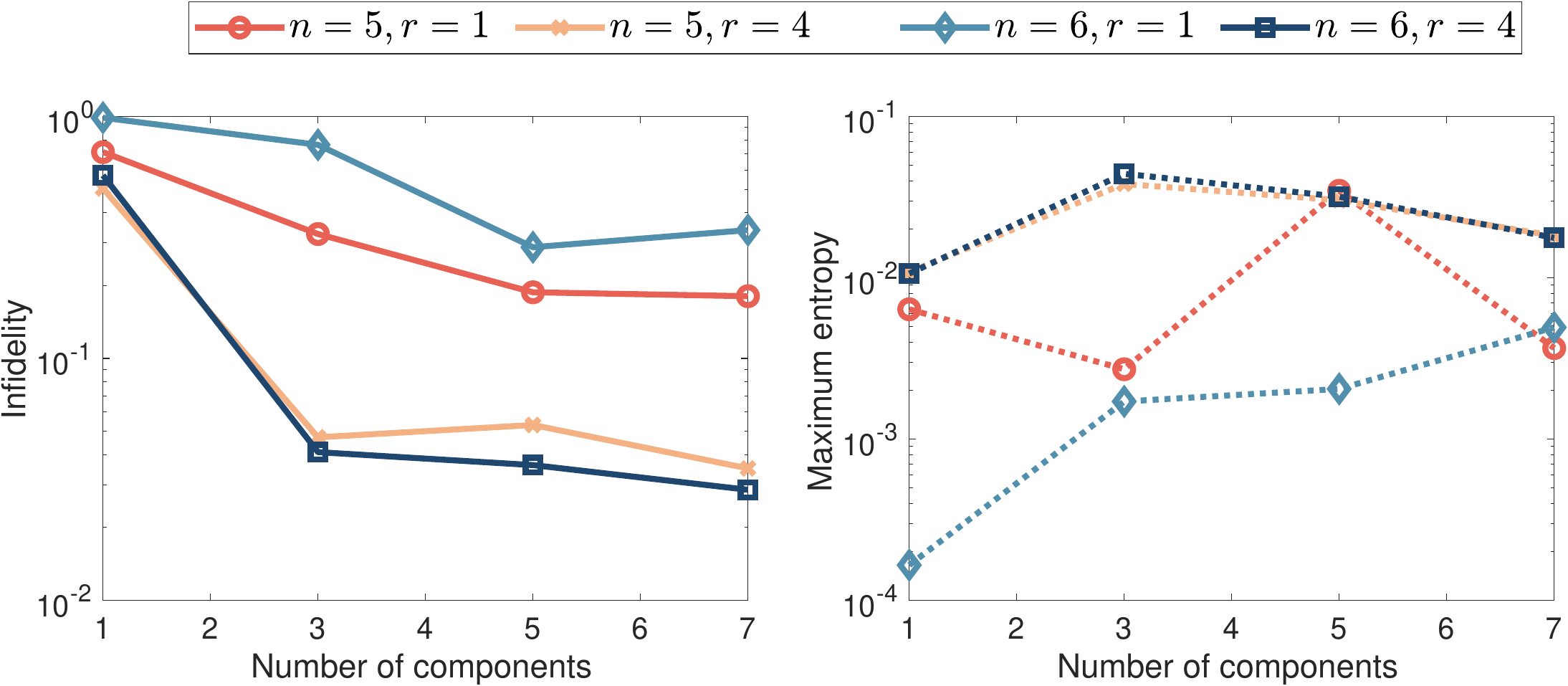}
    \caption{Numerical results on estimating approximate stabilizer rank of $\ket{H^{\ox n}}$ for $n=5,6$ qubits. Left: infidelity. Right: the maximum $2$-stabilizer R\'enyi entropy among each component. }
    \label{fig:stabrank}
\end{figure}

\begin{table}[htbp]
    \centering\scriptsize
    \caption{Numerical results on estimating $(\epsilon,\delta)$-approximate stabilizer rank of $\ket{H^{\ox n}}$ for $n=2,3,4$ qubits, number of components $R=1,2,\dots,\lceil7^{n/6}\rceil$, and parameter $r=1,2$.}
    \label{tab:stabrank}
    \setlength{\tabcolsep}{10pt}       
    \begin{tabular}{ccccl}
        \toprule
        $n$ & $R$ & $r$ & Approximation infidelity & $2$-stabilizer R\'enyi entropy \\
        \midrule
        \multirow{4}*{2} & \multirow{2}*{1} & 1 & 2.4127e-01 & 7.4318e-03\\
         & & 2 & 8.5398e-01 & 5.4383e-03\\
        \cmidrule{2-5}
         & \multirow{2}*{2} & 1 & 1.1871e-01 & 1.5668e-04, 6.4077e-03\\
         & & 2 & 6.9951e-05 & 1.2532e-03, 1.2405e-03\\
        \midrule
        \multirow{6}*{3} & \multirow{2}*{1} & 1 & 4.0758e-01 & 1.1000e-01\\
         & & 2 & 3.3873e-01 & 9.7383e-03\\
        \cmidrule{2-5}
         & \multirow{2}*{2} & 1 & 2.4934e-01 & 4.4777e-03, 9.9008e-04\\
         & & 2 & 2.0903e-01 & 2.4321e-03, 1.2648e-02\\
        \cmidrule{2-5}
         & \multirow{2}*{3} & 1 & 1.1337e-01 & 2.9432e-02, 1.6710e-02, 8.9772e-04\\
         & & 2 & 2.8777e-03 & 5.6882e-03, 4.4175e-03, 3.8043e-03\\
        \midrule
        \multirow{10}*{4} & \multirow{2}*{1} & 1 & 4.3528e-01 & 8.5178e-03\\
         & & 2 & 8.9291e-01 & 4.1529e-03\\
        \cmidrule{2-5}
         & \multirow{2}*{2} & 1 & 4.2484e-01 & 9.0938e-05, 9.7877e-03\\
         & & 2 & 1.0962e-03 & 6.5625e-03, 7.5606e-03\\
        \cmidrule{2-5}
         & \multirow{2}*{3} & 1 & 1.0278e-01 & 2.5400e-04, 3.5386e-03, 7.1955e-03\\
         & & 2 & 1.1394e-03 & 8.2699e-03, 5.7145e-04, 5.5453e-03\\
        \cmidrule{2-5}
         & \multirow{2}*{4} & 1 & 2.6224e-01 & 6.2629e-05, 8.8638e-04, 2.1451e-03, 1.7844e-03\\
         & & 2 & 2.4029e-04 & 1.5270e-03, 1.7929e-06, 1.6754e-03, 1.4706e-03\\
        \bottomrule
    \end{tabular}
\end{table}

\subsection{Minimum output R\'enyi $p$-entropy}\label{subsec:Renyi}
Given a quantum channel $\tensN_{A\to B}$ from $\cL(\cH_A)$ to $\cL(\cH_B)$, the minimum output R\'enyi $p$-entropy of $\tensN_{A\to B}$ is defined by 
\begin{equation}
    \label{eq:p-renyi_entropy}
    S_p^{\min}(\tensN_{A\to B}):=\min_{\rho_A}\frac{1}{1-p}\log\tr(\tensN_{A\to B}(\rho_A)^p),\quad p\in(0,1)\cup (1,+\infty),
\end{equation}
where the minimization is taken over all density matrices $\rho_A$ in $\cL(\cH_A)$. When $p$ goes to $1$, we have $S_1^{\min}(\tensN) = \min_{\rho_A}\!H(\cN(\rho_A))$ where $H(\rho)=-\tr(\rho\log\rho)$ is the von Neumann entropy. It follows from the concavity of the R\'enyi entropies that the minimum~\eqref{eq:p-renyi_entropy} is attained at a pure input state $\rho_A=\ket{\psi}\langle\psi|$. Thus, calculating $S_p^{\min}(\tensN_{A\to B})$ for a given quantum channel can be interpreted as optimization on a unit sphere. One crucial problem in terms of the minimum output R\'enyi $p$-entropy of a given quantum channel $\cN_{A\to B}$ is its \textit{strict subadditivity}, i.e., whether it holds \[S_p^{\min}(\tensN_{A\to B}^{\ox n}) < nS_p^{\min}(\tensN_{A\to B})\] 
for some $n$, given that the subadditivity always holds. A direct numerical method to tackle this problem is to compute $n^{-1}S_p^{\min}(\tensN_{A\to B}^{\ox n})$ for large $n$ and compare it with the one-shot value $S_p^{\min}(\tensN_{A\to B})$. However, the parameter space of the optimization problem for $S_p^{\min}(\tensN_{A\to B}^{\ox n})$ scales exponentially. In view of the tensor product structure of $\cN^{\ox n}_{A\to B}$, we consider the minimization of R\'enyi $p$-entropy on low-rank NTT tensors, i.e.,
\begin{equation}
    \label{eq:low-rank_p-renyi}
    \begin{aligned}
        \min_{\ket{\psi}} & \quad f^{(n)}(\ket{\psi})=\frac{1}{1-p}\log\tr(\cN^{\ox n}_{A\to B}(\ket{\psi}\langle\psi|)^p)\\ 
        \subjectto & \quad\ket{\psi}=\rmvec(\llbracket\tensU_1,\tensU_2,\dots,\tensU_n\rrbracket),\quad\llbracket\tensU_1,\tensU_2,\dots,\tensU_n\rrbracket\in\tensN_\vecr. 
    \end{aligned}
\end{equation}
For $p=2$, we obtain that
\begin{equation*}
    \begin{aligned}
        &~~~~\tr(\tensN_{A\to B}^{\ox n}(\ket{\psi}\langle\psi|)^2)
        &=\sum_{k_1^{(1)},k_2^{(1)},\dots,k_n^{(1)}=1}^K\sum_{k_1^{(2)},k_2^{(2)},\dots,k_n^{(2)}=1}^K \Big|\langle\psi|\bigotimes_{j=1}^n(\matK_{k_j^{(1)}}^\dagger\matK_{k_j^{(2)}}^{})\ket{\psi}\Big|^2,
    \end{aligned}
\end{equation*}
where $\{\matK_{k_j}\}_{k_j}$ is a set of Kraus operators of $\tensN_{A\to B}$. Following the same spirit in~\cite{Haug_2023}, the cost function $f(\ket{\psi})$ can be efficiently computed by the Frobenius norm of a tensor in the TT format with size $2^2\times 2^2\times\cdots\times 2^2$ and bond dimension $(1,r_1^2,r_2^2,\dots,r_{n-1}^2,1)$. We consider two typical channels, the antisymmetric channel and the generalized amplitude damping channel, in numerical experiments.

\paragraph{Antisymmetric channel}
The antisymmetric subspace $\mathrm{asym}_p^d$ is the subspace of $(\CC^d)^{\ox p}$ defined by $\mathrm{asym}_d^p \coloneqq \{ \ket{\psi} \in (\CC^d)^{\ox p} : \ket{\psi} = (-1)^{{\rm sgn}(\sigma)}P_\sigma \ket{\psi},~\text{for all}~\sigma \in S_p \}$, where $S_p$ is the symmetric group, ${\rm sgn}(\sigma)$ is the parity of the permutation $\sigma$, and $P_\sigma$ is the unitary operator that permutes the $p$ subsystems according to the permutation $\sigma$, i.e., $P_\sigma(\ket{\psi_1} \ox \cdots \ox \ket{\psi_p}) \coloneqq \ket{\psi_{\sigma(1)}} \ox \cdots \ox \ket{\psi_{\sigma(p)}}$. When $d=3, p=2$, we have the following basis for the antisymmetric subspace $\ket{\psi_1} = (\ket{01} - \ket{10})/\sqrt{2},~\ket{\psi_2} = (\ket{02} - \ket{20})\sqrt{2},~\ket{\psi_3} = (\ket{12} - \ket{21})\sqrt{2}$. Then consider a channel $\cN_{\mathrm{as}}(\cdot)$ such that its Stinespring isometry is $V:\CC^3 \to \mathrm{asym}_3^2$, which has a matrix form $V=(\ket{\psi_1},\ket{\psi_2},\ket{\psi_3})$. This channel has Kraus operators $\matK_1,\matK_2,\matK_3\in\mathbb{C}^{3\times 3}$ with
\begin{equation*}
    \matK_1 = -\frac{1}{\sqrt{2}}(\ketbra{1}{0} +\ketbra{2}{1}),~\matK_2=\frac{1}{\sqrt{2}}(\ketbra{0}{0}-\ketbra{2}{2}),~\matK_3 = \frac{1}{\sqrt{2}}(\ketbra{0}{1}+\ketbra{1}{2}).
\end{equation*}
We set $n=11,12,\dots,15$ and $\vecr$ with $r=1,2,\dots,10$ in~\cref{eq:rank_params}. The NTT-RCG method is applied to~\cref{eq:low-rank_p-renyi} and terminates if the number of iterations reaches 2000. For each combination of qubit and rank parameter, the NTT-RCG method is tested for five runs.

Numerical results on the antisymmetric channel are reported in~\cref{fig:entropy_anti}. We observe from~\cref{fig:entropy_anti} (left) and~\cref{fig:entropy_anti} (middle) that the average time per iteration scales polynomially with respect to the qubit and rank parameter. We also conclude that the additivity holds for $10,11,\dots,16$ qubits and the NTT tensors with rank parameter no larger than $10$. 

\begin{figure}[htbp]
    \centering
    \includegraphics[width=\linewidth]{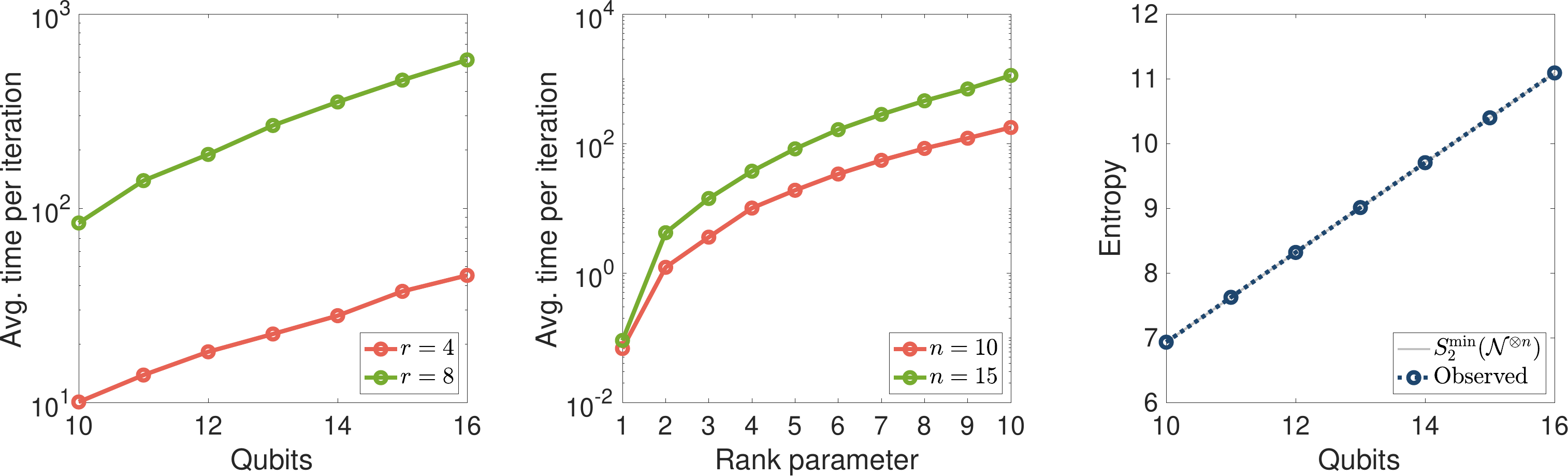}
    \caption{Numerical results on the antisymmetric channel. Left: average time per iteration with respect to qubit. Middle: average time per iteration with respect to the rank parameter. Right: smallest entropy computed from NTT-RCG.}
    \label{fig:entropy_anti}
\end{figure}

\paragraph{Generalized amplitude damping channel (GADC)} The generalized amplitude damping channel is defined by $\cA_{\gamma,N}:\rho\mapsto\sum_{k=1}^4 \mata_k^{}\rho \mata_k^\dag$, where the Kraus operators $\mata_1,\mata_2,\mata_3,\mata_4\in\mathbb{C}^{2\times 2}$ are given by
\begin{align*}
    \mata_1 &= \sqrt{1-N} \big( \ketbra{0}{0} + \sqrt{1-\gamma} \ketbra{1}{1} \big), &
    \mata_2 &= \sqrt{\gamma(1-N)} \ketbra{0}{1}, \\
    \mata_3 &= \sqrt{N} \big( \sqrt{1-\gamma} \ketbra{0}{0} + \ketbra{1}{1} \big), &
    \mata_4 &= \sqrt{\gamma N} \ketbra{1}{0},
\end{align*}
and $\gamma,N\in[0,1]$. We set $n=2,3,\dots,12$ and $\vecr$ with $r=1,2,\dots,10$ in~\cref{eq:rank_params}.
Numerical results on the GADC are reported in~\cref{fig:entropy_gen}. We observe from~\cref{fig:entropy_gen} (left) and~\cref{fig:entropy_gen} (middle) that the average time per iteration scales polynomially with respect to the qubit $n$ and parameter $r$. We also conclude that the additivity holds for $2,3,\dots,12$ qubits and the NTT tensors with rank parameter no larger than~$10$.

\begin{figure}[htbp]
    \centering
    \includegraphics[width=\linewidth]{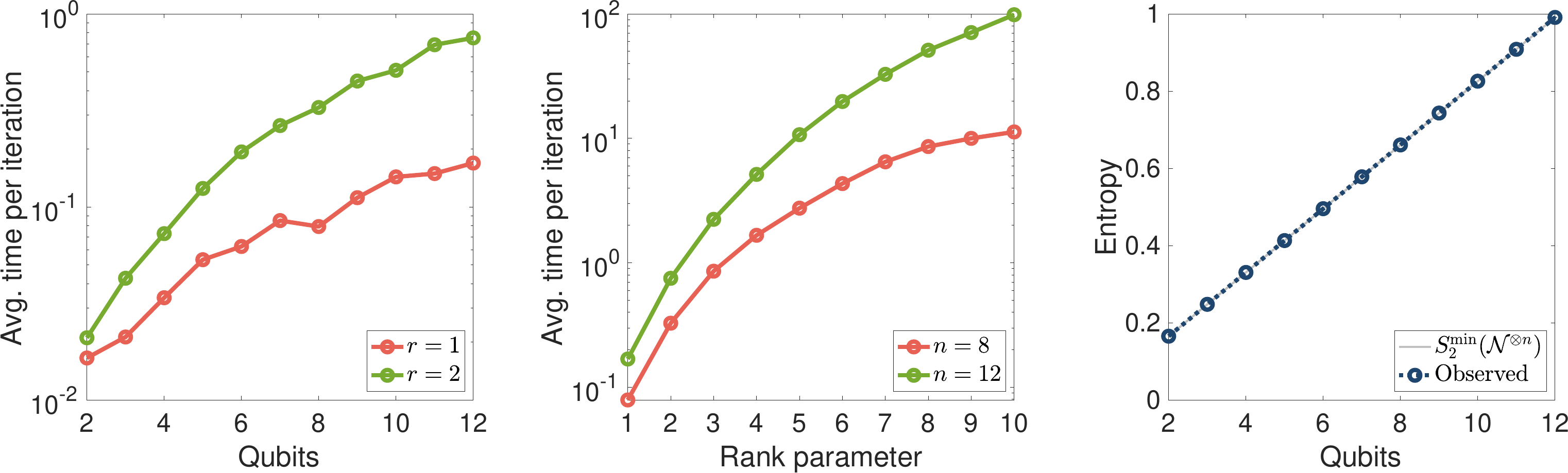}
    \caption{Numerical results on generalized amplitude damping channel. Left: average time per iteration with respect to qubit. Middle: average time per iteration with respect to the rank parameter. Right: smallest entropy computed from NTT-RCG.}
    \label{fig:entropy_gen}
\end{figure}

\section{Conclusion}\label{sec:conclusion}
We have introduced and studied the normalized tensor-train (NTT) format for representing and optimizing over low-rank tensors. We have shown the existence of NTT decomposition for unit-norm tensors and constructed a quasi-optimal approximation operator, which is called the NTT-SVD algorithm. We have proved that the set of fixed-rank NTT tensors forms a smooth manifold and have developed the geometric tools. While NTT and TT decompositions share many similarities in practice, their underlying geometries are fundamentally distinct. NTT naturally encodes tensors on the unit sphere, making it especially suitable for problems where normalization is intrinsic, such as recovery of low-rank tensors with unit norm, eigenvector computations, approximation of the stabilizer rank, computation of the minimum output R\'enyi $p$-entropy, and other applications constrained by unit norms.

We have demonstrated the advantages of the NTT decomposition in several applications. The NTT-based Riemannian conjugate gradient method has better efficiency in eigenvalue and eigenvector computation compared with the single-site DMRG method for large-scale local Hamiltonians. In quantum information theory, geometric methods based on NTT decomposition provide a practical numerical tool to estimate the stabilizer rank of magic states and to search for additivity violations of quantum channel entropy. These problems are intractable with conventional methods due to the exponential scaling of the Hilbert space.

There are some potential directions for future work. In the eigenvalue problem, we have primarily focused on computing a single eigenvector, which has been well-studied in the literature. Extending the NTT framework to efficiently compute multiple eigenvectors or eigenspaces remains an interesting open question. Additionally, we have focused on the unit-norm constraint. It would be worthwhile to explore whether similar frameworks can be developed for tensors subject to other structural constraints, in analogy to recent advances for low-rank matrices~\cite{yang2025space}.



\bibliographystyle{siamplain}
\bibliography{ref}
\end{document}